\documentclass[english, 11pt]{amsart}
\usepackage{amssymb}
\usepackage{amsthm}
\usepackage{amsfonts}
\usepackage{amscd}
\usepackage{mathrsfs}
\usepackage[T1]{fontenc}
\usepackage{enumerate}%
\usepackage{amsthm,lipsum}
\usepackage{tikz}
\usepackage[colorlinks, linkcolor=red]{hyperref}
\usepackage{color}
\usepackage[utf8]{inputenc}

\definecolor{immi}{rgb}{0,.6,.1}

\newbox\removebox
\newcommand\remove[2]{%
\setbox\removebox=\ifmmode\hbox{$#2$}\else\hbox{#2}\fi%
\leavevmode
\rlap{\textcolor{#1}{\vrule height0.8ex depth-0.6ex width\wd\removebox}}%
\box\removebox
}
\long\def\bigremove#1{%
\par\setbox\removebox=\vbox{#1}%
\vbox{%
\vbox to0pt{\hbox{\tikz\draw[color=blue,thick] (0,0) -- (\wd\removebox,-\ht\removebox)  (\wd\removebox,0) -- (0,-\ht\removebox);}}
\box\removebox
}
}

\usepackage{mathrsfs} 

\makeatletter
\def\RFss@@#1{\RF^*_{\!*#1}}
\def\RFss@_#1{\RFss@@{,#1}}
\def\RFss{\@ifnextchar_{\RFss@}{\RFss@@{}}}
\makeatother

\newcommand{\RF}{{\rm RF}}

\def\Supp{{\operatorname{Supp}}}

\def\lct{\operatorname{lct}}
\def\moi{\operatorname{moi}}

\def\Tr{\operatorname{Tr}}

\def\Sing{\operatorname{Sing}}

\def\deg{\operatorname{deg}}
\def\ac{{\overline{\rm ac}}}

\def\11{{\mathbf 1}}
\def\AA{{\mathbb A}}

\def\CC{{\mathbb C}}

\def\FF{{\mathbb F}}

\def\NN{{\mathbb N}}

\def\PP{{\mathbb P}}
\def\QQ{{\mathbb Q}}
\def\RR{{\mathbb R}}

\def\TT{{\mathbb T}}

\def\ZZ{{\mathbb Z}}

\def\cI{{\mathcal I}}

\def\cL{{\mathcal L}}
\def\cM{{\mathcal M}}
\def\cN{{\mathcal N}}
\def\cO{{\mathcal O}}

\def\cT{{\mathcal T}}

\def\cZ{{\mathcal Z}}

\newcommand{\grad}{\operatorname{grad}}

\newcommand{\tr}{\operatorname{Tr}}
\newcommand{\spec}{\operatorname{Spec}}

\usepackage{geometry}
\newtheorem{thm}[subsubsection]{Theorem}
\newtheorem{lem}[subsubsection]
{Lemma}
\newtheorem{cor}[subsubsection]
{Corollary}
\newtheorem{prop}[subsubsection]
{Proposition}
\newtheorem{conj}
{Conjecture}
{Problem}
{Property}
\theoremstyle{plain}
\newtheorem*{namedthm}{\namedthmname}
\newcounter{namedthm}

\makeatletter
\newenvironment{named}[1]
  {\def\namedthmname{#1}%
   \refstepcounter{namedthm}%
   \namedthm\def\@currentlabel{#1}}
  {\endnamedthm}
\makeatother
\theoremstyle{definition}
\newtheorem{defn}[subsubsection]
{Definition}

\theoremstyle{remark}
{Remark}
\newtheorem{rem}[subsubsection]
{Remark}
{Remarks}

\theoremstyle{plain}

  {\par\medskip\noindent #1\par\begingroup%
    \advance\leftskip by 1em\advance\rightskip by 1em}%
  {\par\endgroup}

\DeclareMathOperator*{\Spec}{Spec}

\newcommand{\ord}{\operatorname{ord}}

\def\cI{\mathcal{I}}

\def\cL{\mathcal{L}}
\def\cO{\mathcal{O}}
\def\cT{\mathcal{T}}

\renewcommand{\phi}{\varphi}
\renewcommand{\epsilon}{\varepsilon}
\renewcommand{\theta}{\vartheta}

\renewcommand{\and}{ \quad \text{and} \quad }

\begin{document}

\setcounter{tocdepth}{1} 

\author[Kien~H.~Nguyen]
{Kien Huu Nguyen}
\address{KU Leuven, Department of Mathematics,
Celestijnenlaan 200B, B-3001 Leu\-ven, Bel\-gium
}
\email{kien.nguyenhuu@kuleuven.be}
\address{Institute of Mathematics, Vietnam Academy of Science and Technology, 18 Hoang Quoc Viet, Nghia Do, Hanoi, Vietnam}
\email{nhkien@math.ac.vn}

\address{Normandie Université, Université de Caen Normandie - CNRS, Laboratoire de Mathématiques Nicolas Oresme (LMNO), UMR 6139, 14000 Caen, France}

\urladdr{https://sites.google.com/site/nguyenkienmath/home}

%
%

\begin{abstract}Let $f$ be a polynomial of degree $d>1$ in $n$ variables over $\ZZ$. Let $f_d$ be the homogeneous part of degree $d$ of $f$ and $s$ be the dimension of the critical locus of $f_d$. In this paper, we prove Igusa's conjecture for exponential sums with the exponent $(n-s)/(2(d-1))$. This implies a weak solution for a recent conjecture raised by Cluckers and the author (2020) about an analogue of the results of Deligne (1974) and Katz (1999) for exponential sums over finite fields in the finite ring setting. Moreover, this also improves the result of Cluckers, Musta\c{t}\u{a} and the author (2019) in case $n-s>2(d-1)$. In an extra aspect, this result improves the conditions $n-s>2^d(d-1)$ of Birch (1962) and $n-s>3(d-1)2^{d-2}$ of Browning-Prendiville (2017) on the validity of the estimation for the major arcs associated with $f$ to $(n-s)>4(d-1)$. Therefore this result may have further applications on topics related to the Hardy-Littlewood circle method such as the Hasse principle and distribution of rational points in algebraic varieties.  On the other hand, we also improve the recent work of Cluckers, Koll\'ar and  Musta\c{t}\u{a} (2019) on the strong monodromy conjecture  in the range $(-\lct((f)+J_f^2),0]$ in case of bad reduction and bad Schwartz-Bruhat function. Namely, in the range $(-\lct((f)+J_f^2),0]$, the real part of any pole of the Igusa local zeta functions associated with $f$ and any Schwartz-Bruhat function over any $p$-adic field is a root of the Bernstein-Sato polynomial of $f$.
\end{abstract}

\title[Exponential sums modulo $p^m$ for {D}eligne polynomials]{On a uniform bound for exponential sums modulo $p^m$ for {D}eligne polynomials} 


\maketitle

\section{Introduction}\label{sec:intro}


Let $f(x)$ be a polynomial in $n$ variables of degree $d>1$ over a field $\mathsf{F}$ such that $d$ is invertible in $\mathsf{F}$.   We denote by $s(f)$ the dimension of the critical locus of $f_d:\overline{\mathsf{F}}^n\to \overline{\mathsf{F}}$, where $\overline{\mathsf{F}}$ is a fixed algebraical closure of $\mathsf{F}$ and $f_d$ is the degree $d$ homogeneous part of $f$. When $f_d$ is smooth, i.e, $s(f)=0$, we say that $f$ is a Deligne polynomial. This name follows from the famous theorem of Deligne in \cite{DeligneWI} for exponential sums over finite fields. Let us recall this theorem. 
\begin{thm}\label{Deligne}
Suppose that $f\in\ZZ[x_1,...,x_n]$ is a Deligne polynomial. Let $p$ be a large enough prime and  $k$ be a finite field of characteristic $p$. Then  $$||k|^{-n}\sum_{x\in k^n}\psi(\tr_{k/\FF_p}(f(x)))|\leq (d-1)^n|k|^{-n/2},$$
where $\tr_{k/\FF_p}: k\to \FF_p$ is the trace function and $\psi:(\FF_p,+)\to(\CC^{\times},\times)$ is any non-trivial group homomorphism.
\end{thm} 
\begin{rem}
In the original statement, Deligne worked with a fixed prime $p$ and a Deligne polynomial $f\in\FF_p[x_1,...,x_n]$ of degree $d$. Theorem \ref{Deligne} follows from the fact that if $f\in\ZZ[x_1,...,x_n]$ is a Deligne polynomial then $f\mod p$ is a Deligne polynomial if $p$ is large enough (see, e.g., \cite[Part 3, Theorem 9.7.7]{EGA4}).
\end{rem}
In \cite{Katz}, Katz extended Deligne's theorem to a more general setting where we allows $x$ running in the set of $k$-points of any $k$-scheme $Y$ with good enough geometric condition and $f$ is a regular function on $Y$. Katz's result depends on the complexity of $Y$ and the singularities of $f$ at infinity. Where the complexity of $Y$ depends on a given embedding $Y\hookrightarrow \PP^N$, the number $R$ and the degrees $D_1,...,D_R$ of equations defining $Y$ in $\PP^N$. However, if we only work on the setting of Theorem \ref{Deligne} but allow  $s(f)>0$ then we can get a simple bound of the exponential sum by using Theorem \ref{Deligne} and induction on $s(f)$ in a suitable manner.
\begin{prop}{\cite[Section 3.5]{CluckerNguyen}}\label{Deligne-Katz} Let $f\in\ZZ[x_1,...,x_n]$ be a polynomial of degree $d>1$ and $s(f)$ be given as above. Let $p$ be a large enough prime and  $k$ be a finite field of characteristic $p$ then  $$||k|^{-n}\sum_{x\in k^n}\psi(\tr_{k/\FF_p}(f(x)))|\leq (d-1)^{n-s(f)}|k|^{-(n-s(f))/2}$$
for all non-trivial group homomorphisms $\psi:(\FF_p,+)\to(\CC^{\times},\times)$.
\end{prop}

A natural question that one may ask is whether an analogue of Proposition \ref{Deligne-Katz} exists for exponential sums over finite rings $\ZZ/N\ZZ$, i.e., the exponential sums
$$E_{f}(\psi)=\frac{1}{N^n}\sum_{x\in (\ZZ/N\ZZ)^n} \psi(f(x)),$$
where $N\neq 0$ and $\psi$ is any group monomorphism from $\ZZ/N\ZZ$ to $\CC^{\times}$. More general, we also ask this question for exponential sums
$$E_{f}(\psi)=\frac{1}{|\cO_K/I|^n}\sum_{x\in (\cO_K/I)^n} \psi(f(x)),$$
where $\cO_K$ is the ring of integers in a number field $K$, $I$ is any non-zero ideal of $\cO_K$ and $\psi$ is any group monomorphism from $\cO_K/I$ to $\CC^{\times}$. This question lies in the context of Igusa's conjecture for exponential sums and its relation with the Hasse principle as seen later. Inspired by the strong monodromy conjecture (see Subsection \ref{setup1} for the precise statement), Igusa's conjecture for exponential sums (see, e.g., \cite{Igusa3,CMN}) and the recent work on the minimal exponent of Musta{\c{t}}{\v{a}} and Popa in \cite{MustPopa}, Cluckers and the author proposed in \cite{CluckerNguyen} a conjecture as follows:
\begin{conj}\label{conj1}
Given $f$, $n$, $s(f)$, and $d$ as in Proposition \ref{Deligne-Katz}. For each $\varepsilon>0$, there is a constant $C_{f,\epsilon}>0$ such that
\begin{equation}\label{bound:SfN}
|E_f(\psi)|\leq C_{f,\epsilon}  N^{-\frac{n-s(f)}{d}+\varepsilon}.
\end{equation}
for all integers $N\geq 1$ and all group monomorphisms $\psi:(\ZZ/N\ZZ,+)\to(\CC^{\times},\times)$.
\end{conj}
By the Chinese Remainder Theorem, if one writes $N=\prod_{1\leq i\leq r}p_i^{m_i}$ for distinct prime numbers $p_i$ and integers $m_i>0$, then  for each group monomorphism $\psi:(\ZZ/N\ZZ,+)\to(\CC^{\times},\times)$, there is a group monomorphism $\psi_i:(\ZZ/p_i^{m_i}\ZZ,+)\to(\CC^{\times},\times)$  for each $1\leq i\leq r$ such that
\begin{equation}\label{eq:p-i}
E_{f}(\psi) = \prod_{1\leq i\leq r} E_f(\psi_i)
\end{equation}
By (\ref{eq:p-i}) and Proposition \ref{Deligne-Katz}, to prove Conjecture \ref{conj1}, it is sufficient to show that for each $\epsilon>0$, there is a constant $C_{f,\epsilon}$ such that for all primes $p$, all $m>0$ and all group monomorphisms $\psi:(\ZZ/p^m\ZZ,+)\to(\CC^{\times},\times)$ one has
\begin{equation}\label{allp}
|E_{f}(\psi)|\leq C_{f,\epsilon}p^{m(-\frac{n-s(f)}{d}+\epsilon)},
\end{equation}
in addition, when $p$ is large enough and $m>1$, one also needs
\begin{equation}\label{largep}
|E_{f}(\psi)|\leq p^{m(-\frac{n-s(f)}{d}+\epsilon)}.
\end{equation}

Inequality (\ref{largep}) for large $p$ and $m>1$ relates to Igusa's conjecture for exponential sums. Let us explain this in more detail. Let $p$ be a prime and $m>0$ be an integer. Let $\psi:(\ZZ/p^m\ZZ,+)\to (\CC^{\times},\times)$ be a group monomorphism. There exists an additive character $\tilde{\psi}$ of $\QQ_p$ of conductor $m$, i.e., $\tilde{\psi}:(\QQ_p,+)\to (\CC^{\times},\times)$ is a continuous homomorphism of groups, 
$\tilde{\psi}|_{p^{m}\ZZ_p}=1$ and $\tilde{\psi}|_{p^{m-1}\ZZ_p}\neq 1$,  such that
\begin{equation}\label{rephrase}
E_f(\psi)=E_f(\tilde{\psi}):=\int_{\ZZ_p^n}\tilde{\psi}(f(x))|dx|,
\end{equation}
where $|dx|$ is the Haar measure on $\QQ_p^n$ such that the volume of $\ZZ_p^n$ is $1$. The integral in (\ref{rephrase}) helps us to generalize the notion of exponential sums to all non-Archimedean local fields. We refer to Subsection \ref{setup1} for all details, here we sketch the main points.  Let $K$ be a number field and $\cO_K$ be its ring of integers. Let $Y$ be  an $\cO_K$-subscheme of $\AA^n_{\cO_K}$ of finite type. Let $L$ be a non-Archimedean local field over $\cO_K$, i.e, $L$ is endowed with a structure of $\cO_K$-algebra. We associate to $Y$ the function $\Phi_{Y,L}:=\textbf{1}_{\{x\in \cO_L^n|\overline{x}\in Y(k_L)\}}$ on $L^n$, where $\cO_L$ is the ring of integers in $L$, $k_L$ is the residue field of $L$ and $\overline{x}$ is the reduction of $x$ modulo the maximal ideal $\cM_L$ of $\cO_L$. Let $f\in\cO_K[x_1,...,x_n]$ and $\psi:(L,+)\to(\CC^{\times},\times)$ be a non-trivial additive character of $L$.  We associate to $f,\psi,Y$ the exponential sum 
$$E_{f}(Y,\psi)=\int_{\cO_L^n}\Phi_{Y,L}\psi(f(x))|dx|.$$
When $Y=\AA^n_{\cO_K}$, we write $E_f(\psi)$ instead of $E_{f}(\AA^n_{\cO_K},\psi)$. Now we recall a simple version of Igusa's conjecture for exponential sums.
\begin{conj}\label{conj2} Let $f\in\cO_K[x_1,...,x_n]$ be a non constant polynomial. Let $Y$ be an $\cO_K$-subscheme of $\AA^n_{\cO_K}$ of finite type. Let $\sigma$ be a positive real number. Suppose that for each non-Archimedean local field $L$ over $\cO_K$ of large enough residue field characteristic, there is a constant $c_L$ such that 
\begin{equation}\label{Igu-conj}
|E_f(Y,\psi)|\leq c_L|k_L|^{-m_{\psi}\sigma}
\end{equation} 
for all additive characters $\psi$ of $L$ of conductor $m_\psi>1$ (see Subsection \ref{setup1} for the definition of $m_\psi$). Then, in the inequality (\ref{Igu-conj}), one can take $c_L$ independent of $L$ provided that $k_L$ is of large enough characteristic.
\end{conj}
\begin{rem}
The supremum taken over all $\sigma$ satisfying the hypothesis of Conjecture \ref{conj2} is the motivic oscillation index $\moi_{K,Y}(f)$ of $f$ at $Y$ over $K$ (see \cite{CMN,NguyenVeys}). This index relates to the largest non-trivial pole of the Igusa local zeta functions (see Subsection \ref{setup1} for the definition of these functions) associated with $f,Y,L$ when the local field $L$ is of large enough residue field characteristic. If $f-a$ has non-rational singularities at $Y(\CC)$ for some critical value $a$ of $f$ then Cluckers, Musta{\c{t}}{\v{a}} and the author proved Conjecture \ref{conj2} for any $\sigma<\moi_{K,Y}(f)=\min_{a\in Y(\CC)}\lct_{a}(f-f(a))$ in \cite{CMN}, where  $\lct_a(g)$ is the log canonical threshold of $g$ at $a$ (see Subsection \ref{setup2} for the definition of log canonical thresholds). Recently, Cluckers, Koll\'ar and  Musta{\c{t}}{\v{a}} showed in \cite{CKMu} that $\min_{a\in Y(\CC)}\lct_a((f-f(a))+J_f^2)$ is a lower bound of $\moi_{K,Y}(f)$, where $J_f$ is the Jacobian ideal of $f$ and $\lct_a((f-f(a))+J_f^2)$ is the log canonical threshold of the ideal $(f-f(a))+J_f^2$ at $a$.

 Within Inequality (\ref{largep}), we expect that $\sigma=(n-s(f))/d$ satisfies the hypothesis of Conjecture \ref{conj2} for $Y=\AA_\ZZ^n$. In other word, we should have $\moi_{\QQ,\AA_\ZZ^n}(f)\geq (n-s(f))/d$. However,  this seems to be a very hard problem at this moment. 
\end{rem}
\begin{rem}
Suppose that one can show (\ref{largep}) for $p>M$ and $m>1$. To prove Conjecture \ref{conj1}, it remains to prove (\ref{allp}) for $p\leq M$. By the work of Igusa in \cite{Igusa3}, this question leads us to the study of non-trivial poles of the Igusa local zeta functions of $f$ over $p$-adic fields for  small primes $p$  such that we do not know the existence of a log-resolution of $(f)$ with good reduction modulo $p$ (see Subsection \ref{setup2}). Until present, there is no effective tool to deal with this question. Indeed, the main strategy coming from the work of Igusa in \cite{Igusa3} is to find a log-resolution $\pi$ of $(f)$ and compute candidate poles of the Igusa local zeta functions of $f$ in terms of the numerical data $(\nu_i,N_i)_{i\in\cT}$ associated with $\pi$ (see Subsection \ref{setup2}). More precisely, it follows from \cite{Igusa3} that the real part of any non-trivial pole of any Igusa local zeta function of $f$ is of the form $\displaystyle\dfrac{-\nu_i}{N_i}$ with $(\nu_i,N_i)\neq (1,1)$. Then, it requires that $\min_{i\in\cT, (\nu_i,N_i)\neq (1,1)}\nu_i/N_i$ is at least $(n-s)/d$ to derive (\ref{allp}) for small $p$ . However, it is not known how to construct such a log-resolution with this property in the general case. 
\end{rem}
In this paper we will prove a weaker version of Conjecture \ref{conj1} for $(n-s(f))/(2(d-1))$ instead of $(n-s(f))/d$. First of all, we will show that  Conjecture \ref{conj2} holds for $\sigma\leq (n-s(f))/(2(d-1))$ as follows. 
\begin{named}{Theorem A}\label{generalZ2} Let $K$ be a number field and $f\in\cO_K[x_1,...,x_n]$ be a non-constant polynomial. Suppose that $Y$ is an $\cO_K$-subscheme of $\AA_{\cO_K}^n$ of finite type whose complexity bounded by $(R,D)$, i.e., $Y_K$ can be defined by $R$ polynomials in $n$ variables of degree at most $D$. There exists an integer $M$ depending only on $f,Y$ and a constant $C$ depending only on $n,d,R,D$ such that for all non-Archimedean local fields $L$ over $\cO_K$ of residue field characteristic at least $M$ and all additive characters  $\psi$ of $L$ of conductor $m_\psi\geq 2$ we have
$$|E_{f}(Y,\psi)|\leq C|k_L|^{\frac{-m_\psi(n-s(f))}{2(d-1)}}.$$
\end{named}
The main idea for the proof of \ref{generalZ2} relates to a recent work of Veys and the author in \cite{NguyenVeys}, where Conjecture \ref{conj2} could be viewed as a conjecture for a family of exponential sums over finite fields associated with $Y$. For a general $Y$, this family of exponential sums is very complicated and difficult to deal directly with. In this paper, we overcome this problem by a new strategy. Firstly, we observe that if $Y$ is a good point of $\AA_{\cO_K}^n$ whose residue field is of large enough characteristic then these exponential sums are of quite beautiful form (we say that a point $P$ in a scheme is a good point if its residue field  is finite). More precisely, they are exponential sums of certain weighted homogeneous polynomials on affine spaces over finite fields.  Thus we can use a consequence of Katz's result in \cite{CDenSperlocal} for exponential sums associated with weighted homogeneous polynomials on affine spaces. The main difficulty is to estimate the dimension of the singular locus of these weighted polynomials. Fortunately, we find an interesting connection between these singular loci and the Jet schemes associated with the Jacobian ideal of $f$. This allows us to use Musta{\c{t}}{\v{a}}'s result on the log-canonical threshold of ideals to bound the dimension of these singular loci. By this way, we are able to give a uniform version of \ref{generalZ2} when $Y$ runs over all good points of $\AA_{\cO_K^{\rm alg}}^n$ whose residue fields are of large enough characteristic, where $\cO_K^{\rm alg}$ is the integral closure of $\cO_K$ in $\overline{\QQ}$ (if $L$ is a non-Archimedean local field such that $Y$ is not defined over $\cO_L$ then we put $E_{f}(Y,\psi)=0$ for all additive characters $\psi$ of $L$).  In the next step, we try to
 use this uniform version to obtain our result for arbitrary $Y$ by using the tool of algebraic geometry and Cluckers-Loeser's motivic integration theory. More precisely, we cut out $Y$ by a family of affine spaces $\{a+H\}_{a\in\AA^n}$ for a suitable linear subspace $H$ such that the intersection of $Y$ and each member of this family has only finitely many $\CC$-points (after removing from $Y$ a proper closed subset). The above uniform version of \ref{generalZ2} for good points of $\AA_{\cO_K^{\rm alg}}^{\dim(H)}$ yields a uniform bound for the exponential sums associated with $f|_{a+H}$ and any good point $P\in (Y\cap a+H)_{\cO_K^{\rm alg}}$ whose residue field is of large enough characteristic. Furthermore, if $P$ is a point of $(Y\cap a+H)_{\cO_K^{\rm alg}}$ whose residue field is of large enough characteristic, then we try to relate the exponential sums associated with $f$ and $P$  to the exponential sums associated with $f|_{a+H}$ and $P$  by generalizing an idea from \cite{Cigumodp} about constructible motivic functions and integrals with parameters. \ref{generalZ2} follows from the fact that the exponential sum associated with $f, Y,\psi$ is the sum taken over all good points $P\in Y_{\cO_K^{\rm alg}}$ of the exponential sum associated with $f, P, \psi$. Notice that our strategy can be repeated to improve \ref{generalZ2} whenever we have a good enough uniform version of \ref{generalZ2} for good points $Y$ of $\AA_{\cO_K^{\rm alg}}^n$. In particular, Conjecture \ref{conj1} may follow if one can improve the bound for exponential sums associated with the above weighted homogeneous polynomials on affine spaces. By the work in \cite{Katz}, better understanding of the vanishing cycle of these above weighted homogeneous polynomials may be a key to attack Conjecture \ref{conj1}. 

To show (\ref{allp}) for $(n-s(f))/(2(d-1))$ instead of $(n-s(f))/d$, we will show first that $(n-s(f))/(2(d-1))<\lct_a((f-f(a))+J_f^2)$ for all $a$. After that, we extend the result of Cluckers, Koll\'ar and Musta{\c{t}}{\v{a}} in \cite[Proposition 6.4]{CKMu} to the case of all $p$-adic fields and all locally constant functions with compact support (Schwartz-Bruhat functions).  Namely, we prove that:
\begin{named}{Theorem B}\label{smallprime2}Let $L$ be a finite extension of $\QQ_p$ for a prime number $p$. Suppose that $f$ is a non-constant polynomial in $L[x_1,...,x_n]$. For each  Schwartz-Bruhat function $\Phi$ on $L^n$ and each $0<\sigma<\sigma_0:=\min_{a\in\Supp(\Phi)}\lct_{a}((f-f(a))+J_f^2)$, there is a constant $c(f,\sigma,\Phi)>0$ depending on $f,\sigma,\Phi$ such that for all additive characters $\psi$ of $L$ we have
\begin{equation}\label{apasb}
|E_{f}(\Phi,\psi)|:=|\int_{L^n}\Phi(x)\psi(f(x))|dx||\leq c(f,\sigma,\Phi)|k_L|^{-m_\psi\sigma}.
\end{equation}
\end{named}
As mentioned above, Igusa's idea on the numerical data of log-resolution is hard to use to prove (\ref{apasb}). In fact, we do not know a way to construct a log resolution of $(f)$ with the numerical data $(\nu_i,N_i)_{i\in \cT}$ such that $\min_{i\in\cT, (\nu_i,N_i)\neq (1,1)}\nu_i/N_i\geq \lct_a((f)+J_f^2)$ for all singular points $a$ of $f$. In this paper, we will try to prove (\ref{apasb}) directly, as in \cite[Proposition 6.4]{CKMu} but with a more advanced technique. More precisely, we need to control well twisted Igusa local zeta functions associated with higher ramifield characters. As in \cite[Corollary 6.2]{CKMu}, \ref{smallprime2} also implies a consequence in Section \ref{small} related to the strong monodromy conjecture in the ranges $(-\lct(f,J_f^2),0]$ and $[-(n-s(f))/(2(d-1)),0]$. Moreover, this consequence  extends the result in \cite[Corollary 6.2]{CKMu} to the setting of all $p$-adic fields and all Schwartz-Bruhat functions.

By using \ref{generalZ2} and \ref{smallprime2}, we will show in Section \ref{small} that:
\begin{named}{Theorem C}\label{globalfield}
Let $K$ be a number field and $f$ be a non-constant polynomial in $\cO_K[x_1,...,x_n]$ of degree $d>1$. For each finite extension $K'$ of $K$ and $\epsilon>0$,  there is a constant $c_{f,K',\epsilon}$ depending on $f,K',\epsilon$ such that for all non-zero ideals $I$ of $\cO_{K'}$ and all group monomorphisms $\psi:\cO_{K'}/I\to\CC^{\times}$ we have 
$$|E_{f}(\psi)|\leq C_{f,K',\epsilon}|\cO_{K'}/I|^{-\frac{n-s(f)}{2(d-1)}+\epsilon}$$
\end{named}

As mentioned in \cite{Igusa3,CluckerNguyen}, Conjecture \ref{conj1} will help to obtain a simple condition for the validity of a certain ad\`elic Poisson summation formula and the estimation for the major arcs towards the smooth Hasse principle of $f$. More precisely, we expect that this condition is $n-s(f)>2d$. Furthermore, we may hope that there is an improvement of Birch's result in \cite{Birch} for the smooth Hasse principle related to this condition in the future. 

Note that the current condition for the validity of the estimation for the major arcs relates to Birch's version in \cite{Birch} where he claimed that $n-s(f)>2^d(d-1)$ is a sufficient condition. In \cite{BrowPren}, Browning and Prendiville improved Birch's condition  to $n-s(f)>\frac{3}{4}2^d(d-1)$. In fact, using \ref{globalfield} and the result of Igusa on zeta functions over Archimedean local fields in \cite{Igusa3}, we can improve Birch's condition to $n-s(f)>4(d-1)$ (see Section \ref{major}). Thus although Conjecture \ref{conj1} has not been proved yet, the result of this paper is still good enough for further applications.

In fact, even if we can improve the estimation for the major arcs, it remains a huge challenge to overcome the estimation for the minor arcs to obtain a new form of the smooth Hasse principle (see Section \ref{major}).  However, with the result of this paper, any improvement of the estimation for the minor arcs in the future may automatically imply a new form of the smooth Hasse principle. On the other hand, the estimation for the major arcs appears in many researches in number theory related to distribution of rational points in algebraic varieties under the Hardy-Littlewood circle method. Therefore, the result of this paper may streamline much future research related to this topic.

Lastly, we mention that the idea of Section \ref{bound} is developed  in \cite{MoiIde} to study uniformly point counting in flat families of schemes over finite rings. 
\section{Preliminaries}\label{setup}
In this section, we introduce the notations and notions used throughout this paper.
\subsection{Local fields, exponential sums and Igusa local zeta functions}\label{setup1}Let $R$ be a commutative ring with unit $1\neq 0$ and $\mathfrak{a}$ be an ideal of $R[x_1,...,x_n]$. We denote by $Z(\mathfrak{a})$ the closed subscheme of $\AA_R^n$ associated with $\mathfrak{a}$. Let $f\in R[x_1,...,x_n]$ be a non-constant polynomial. We denote by $J_f$ the Jacobian ideal of $f$ over $R$, i.e., the ideal generated by polynomials $\frac{\partial f}{\partial x_i}$ for $1\leq i\leq n$. We set $C_f:=Z(J_f)$ and call it the $R$-scheme of critical points of $f$. If $f\in R[x_1,...,x_n]$ is a polynomial, we say that $f$ is weighted homogeneous if there exists a tuple $(r_1,...,r_n,r)\in\ZZ_{>0}^{n+1}$ such that $$f(\lambda^{r_1}x_1,...,\lambda^{r_n}x_n)=\lambda^{r}f(x_1,...,x_n)$$ for all $\lambda\in R'$, all $R$-algebras $R'$. If $R$ is a field and  $f$ is weighted homogeneous then $C_f$ is equal to the singular locus $\Sing(f)$ of the $R$-scheme $Z(f)$.   

Let $(R,v)$ be a discrete valuation ring.  Let $\mathfrak{a}$ be an ideal of $R[x_1,...,x_n]$ and $P\in R^n$, we set $v(\mathfrak{a}(P))=\min_{g\in\mathfrak{a}}v(g(P))$.

Let $F$ be a field, we denote by  $\overline{F}$ a fixed  algebraic closure of $F$. Let $K$ be a number field,  we denote by $\cO_K$ the ring of integers in $K$ and $\cO_K^{\rm alg}$ the integral closure of $\cO_K$ in $\overline{K}$. In this paper, a local field over $\cO_K$ is  a finite extension $L$ of $\QQ_p$ or $\FF_p((t))$ for some prime $p$ such that moreover $L$ is equipped with a structure  of $\cO_K$-algebra. For each integer $M$, we denote by $\cL_{K,M}$ the set of all local fields over $\cO_K$ of characteristic zero and of residue field characteristic at least $M$. We denote by $\cL'_{K,M}$ the set of all local fields over $\cO_K$ of characteristic at least $M$ and we put $\tilde{\cL}_{K,M}=\cL_{K,M}\cup\cL'_{K,M}$. 

Let $L$ be a local field over $\cO_K$. We denote by $\cO_L$ and $\cM_L$ the valuation ring of $L$ and its maximal ideal, respectively and by $k_L=\cO_L/\cM_L$ the residue field of $L$. We use $p_L,q_L$ respectively to denote the characteristic and the cardinality of $k_L$. For $x\in L$, we denote by $\ord_L(x) \in \ZZ\cup\{+\infty\}$ its valuation and  by $|x|_L=q_L^{-\ord_L(x)}$ its absolute value. We fix a uniformizing parameter $\varpi_L$ for $\cO_L$, and denote by $ac:L\to\cO_L^{\times}$ the map given by $ac(x)=x\varpi_L^{-\ord_L(x)}$ if $x\neq 0$ and $ac(0)=0$.
For $x\in \cO_L$, we write $\overline{x}$ for the image of $x$ by the residue map $\cO_L\to \cO_L/\cM_L=k_L$, and we put $\ac(x)=\overline{ac(x)}$ if $x\in L$. 

For each positive integer $n$, we endow $L^n$ with  the Haar measure $|dx|$ which is normalized such that the volume of $\cO_L^n$ is $1$.

\medskip
A {\em multiplicative character} $\chi$ of $\cO_L^{\times}$ is a continuous homomorphism $\chi:\cO_L^{\times}\to \CC^{\times}$ with finite image. We also put $\chi(0)=0$. We write $\chi_{\rm triv}$ for the trivial character of $\cO_L^{\times}$, i.e., $\chi_{\rm triv}(z)=1$ for all $z\in\cO_L^{\times}$.
If $\chi$ is a multiplicative character, the {\em conductor} $c(\chi)$ of $\chi$ is the smallest integer $e\geq 1$ such that $\chi$ is trivial on $1+\varpi_L^e\cO_L$. When $c(\chi)=1$ then $\chi$ induces a multiplicative character $\overline{\chi}:k_L^{\times}\to\CC^{\times}$.
An {\em additive character} of $L$ is a continuous homomorphism $\psi:L\to\CC^{\times}$. If $\psi$ is a non-trivial additive character, the {\em conductor} $m_{\psi}$ of $\psi$ is the integer $m$ such that $\psi$ is trivial on $\varpi_L^m\cO_L$, but non-trivial on $\varpi_L^{m-1}\cO_L$. If $\psi$ is the trivial additive character, we put $m_\psi=-\infty$. If we fix an additive character $\psi$ of conductor $0$, then any non-trivial additive character of $L$ is of the form $\psi_z:L\to\CC^{\times}: x \mapsto \psi(zx)$ for some non-zero element $z$ of $L$.  In particular, we have that $m_{\psi_z}=-\ord_L(z)$. 

 Let $\Phi$ be a Schwartz-Bruhat function on $L^n$, i.e., a locally constant function with compact support. We denote by $\Supp(\Phi)$ the support of $\Phi$.  Let $f\in L[x_1,...,x_n]$ be a non-constant polynomial. The {\em exponential sum} associated with $f$, $\Phi$ and an additive character $\psi$ is
\begin{equation}\label{expdef}E_{f}(\Phi,\psi)=\int_{L^n}\Phi(x)\psi(f(x))|dx|.
\end{equation}
 The {\em Igusa local zeta function} associated with $f$, $\Phi$ and a multiplicative character $\chi$ is
$$\cZ_{\chi}(f,L,\Phi,s)=\int_{L^n}\Phi(x)\chi(ac(f(x)))|f(x)|^s|dx|$$
where $s\in \CC$ with $\Re(s)>0$. Suppose that $L\in\cL_{\QQ,1}$. It is well known that $\cZ_{\chi}(f,L,\Phi,s)$ is holomorphic in this region, that it has a
meromorphic continuation to $\CC$, and that it is a rational function in $q_L^{-s}$. We refer the readers to \cite{Igusa3,DenefBour} for basic properties of  Igusa local zeta functions and exponential sums. Such a basic property is that if $f(\Supp(\Phi)\cap C_f(L))=0$ then $\cZ_{\chi}(f,L,\Phi,s)=0$ for all but finitely many characters $\chi$. This will allows us to state Proposition \ref{Fourier} below. 

 Let $K$ be a number field and $f$ be a non-constant polynomial in $K[x_1,...,x_n]$. Then there is an element $P\in K[x_1,\dots,x_n,\frac{\partial}{\partial x_1},\dots, \frac{\partial}{\partial x_n},s]$ and a polynomial $b(s)\in K[s]\setminus\{0\}$ such that $Pf^{s+1}=b(s)f^s$ (see \cite{Bern}). The monic polynomial $b_f(s)\neq 0$ of smallest degree satisfying this property is called the Bernstein-Sato polynomial of $f$.  The strong monodromy conjecture predicts a deep relation between poles of Igusa local zeta functions and roots of the Bernstein-Sato polynomial $b_f(s)$ of $f$.

\begin{conj}[Strong monodromy conjecture]\label{monodormyconj}Let $K$ be a number field and $f\in\cO_K[x_1,...,x_n]$ be a non-constant polynomial. Let $L$ be a local field over $\cO_K$. If $s$ is a pole of $\cZ_{\chi}(f,L,\Phi,s)$ then $\Re(s)$ is a root of the Bernstein-Sato polynomial $b_{f}(s)$ of $f$ provided that $k_L$ is of large enough characteristic.
\end{conj}
\begin{rem} We also expect that Conjecture \ref{monodormyconj} holds for all non-Archimedean local fields instead of non-Archimedean local fields of large enough residue field characteristic. More general, we can state Conjecture \ref{monodormyconj} for any non-constant polynomial $f$ whose coefficients are in a $p$-adic field $L$, and any Schwartz-Bruhat function $\Phi$ on $L^n$.
\end{rem}



We recall the following proposition from \cite{DenefBour}, relating exponential sums to Igusa local zeta functions.

\begin{prop}[\cite{DenefBour}, Proposition 1.4.4]\label{Fourier}
Let $L$ be a local field over $\ZZ$ and $\Phi$ be a Schwartz-Bruhat function on $L^n$.  Let $f\in L[x_1,...,x_n]$ be a non-constant polynomial and $\psi$ be an additive character of $L$ of positive conductor. Suppose that $f(\Supp(\Phi)\cap C_f(L))=0$. We write  $m$ for $m_\psi$ and $t$ for $q_L^{-s}$. There is $u\in\cO_L^{\times}$ such that 
\begin{align*}
E_f(\Phi,\psi)&= \cZ_{\chi_{\rm triv}}(f,L,\Phi,0) + \mbox{\textnormal{Coeff}}_{t^{m-1}}\Big(\dfrac{(t-q_L)\cZ_{\chi_{\rm triv}}(f,L,\Phi,s)} {(q_L-1)(1-t)}\Big)\\
&+\sum_{\chi\neq \chi_{\rm triv} }g_{\chi^{-1},\psi}\chi(u)\mbox{\textnormal{Coeff}}_{t^{m-c(\chi)}}\big(\cZ_{\chi}(f,L,\Phi,s) \big),
\end{align*}
where the summation index $\chi$ runs over all non-trivial multiplicative characters and
\begin{equation}
g_{\chi,\psi}=\frac{q_L^{1-c(\chi)}} {q_L-1}\sum_{\overline{v}\in (\cO_L/\cM_L^{c(\chi)})^{\times}}\chi(v)\psi(\varpi_L^{m-c(\chi)}v  ).
\end{equation}
\end{prop}
For a local field $L$ over $\ZZ$ and a subset $Y$ of $k_L^n$, we set $\Phi_{Y,L}:=\textbf{1}_{\{x\in\cO_L^n| \overline{x}\in Y\}}$. For a subscheme $Y$ of $\AA_{\cO_K}^n$, we associate to $Y$ a family of Schwartz-Bruhat functions $(\Phi_{Y,L})_L$ where $L$ runs over all local fields over $\cO_K$ and $\Phi_{Y,L}:=\Phi_{Y(k_L),L}$.  To simplify notation, if $Y$ is a subscheme of $\AA_{\cO_K}^n$ or $Y\subset k_L^n$ we will write $\cZ_{\chi}(f,L,Y,s)$ and $E_{f}(Y,\psi)$ instead of $\cZ_{\chi}(f,L,\Phi_{Y,L},s)$ and $E_{f}(\Phi_{Y,L},\psi)$ respectively. In addition, if $Y=\{P\}$ for a point $P\in k_L^n$ we also write $\cZ_{\chi}(f,L,P,s)$ and $E_{f}(P,\psi)$ instead of $\cZ_{\chi}(f,L,\{P\},s)$ and $E_{f}(\{P\},\psi)$ respectively.

If $Y$ is a subscheme of $\AA_{\cO_K}^n$ and $\Phi=\Phi_{Y,L}$, we can simplify the formula of Proposition \ref{Fourier} by the following properties.
\begin{prop}\cite[Theorem 2.1]{Denef91}\label{ex-zeta}Let $K$ be a number field and $f\in\cO_K[x_1,...,x_n]$ be a non-constant polynomial. Suppose that $f((Y\cap C_f)(\CC))=0$ then  there exists an integer $M$ depending only on $f$ and a finite set $\cN$ of positive integers such that for all local fields $L\in\tilde{\cL}_{K,M}$ we have
$$\cZ_{\chi}(f,L,Y,s)=0$$
provided that at least one of two following conditions holds
\begin{itemize}
\item[i,]$\chi^N\neq 1$ for all $N\in\cN$,
\item[ii,] $c(\chi)>1,$
\end{itemize}
 and
 $$\cZ_{\chi}(f,L,Y,s)=\cZ_{\chi'}(f,L',Y,s)$$
 if $k_L=k_{L'}$, $c(\chi)=c(\chi')=1$ and $\overline{\chi}=\overline{\chi}'$. Moreover we have
$$E_{f}(Y,\psi)=0$$
for all additive characters $\psi:L\to\CC^{\times}$ of conductor at least $2$ if $p_L>M$ and
$$(Y\cap C_f)(\CC)=\emptyset$$
\end{prop}
\subsection{Log-resolutions of singularities and log-canonical thresholds}\label{setup2}
In this subsection, we recall notations and results related to log-resolutions of ideals. Let $F$ be a field and $\mathfrak{a}$ be a non-zero ideal of $F[x_1,...,x_n]$. A log-resolution of $\mathfrak{a}$ is a projective, bi-rational $F$-morphism $\pi: X'\to X=\AA_F^n$ such that $X'$ is smooth, $\mathfrak{a}\cO_{X'}$ is the ideal of a Cartier divisor $D$ such that $D+K_{X'/X}$ is a divisor with simple normal crossings, where the relative canonical divisor $K_{X'/X}$ is locally defined by the determinant of the Jacobian matrix of $\pi$. This means that at every point $Q\in X'$ there are local coordinates $y_1,...,y_n$ such that $D+K_{X'/X}$ is defined by $(y_1^{d_1}...y_n^{d_n})$ for some $d_1,...,d_n\in\ZZ_{\geq 0}$. We also require that $\pi$ is an isomorphism over the complement of  the zero locus $Z(\mathfrak{a})$ of $\mathfrak{a}$.  The existence of log-resolutions when $F$ has characteristic zero follows from the work of Hironaka (see \cite[page 142, Main Theorem II]{Hir:Res}).

From now on, we suppose that there is a log-resolution $\pi$ of $\mathfrak{a}$ as above. We can write $D=\sum_{i\in\cT}N_iE_i$ and $K_{X'/X}=\sum_{i\in \cT}(\nu_i-1)E_i$, where $\sum_{i\in\cT}E_i$ is a divisor of $X'$ with simple normal crossings and $\nu_i,N_i$ are positive integers for all $i\in\cT$. The set $(\nu_i,N_i)_{i\in\cT}$ will be called the numerical data associated with $\pi$. Let $C$ be a positive integer, we say that the complexity of $(X',\pi)$ is bounded by $C$ if $\#T$, $\nu_i, N_i$ for $i\in T$ and the complexity of $(\sqcup_{i\in T}E_i)\sqcup Y$ are bounded above by $C$. Let $P$ be a point in $Z(\mathfrak{a})(\overline{F})$. The log canonical threshold $\lct_P(\mathfrak{a})$ of $\mathfrak{a}$ at $P$ is given by the following formula:
$$\lct_P(\mathfrak{a})=\min_{i\in\cT, P\in \pi(E_i)(\overline{F})}\frac{\nu_i}{N_i}.$$
This definition is independent of the choice of $\pi$. 
If $P\notin Z(\mathfrak{a})(\overline{F})$, we put $\lct_P(\mathfrak{a})=+\infty$. If $Y$ is a subscheme of $\AA_{F}^n$ then the log canonical threshold of $\mathfrak{a}$ at $Y$, denoted by $\lct_Y(\mathfrak{a})$, is defined to be the minimal element of the set $\{\lct_P(\mathfrak{a})|P\in Y(\overline{F})\}$. To simplify notation, we will write $\lct(\mathfrak{a})$ instead of  $\lct_{Z(\mathfrak{a})}(\mathfrak{a})$. If $R$ is a subring of $F$ and $\mathfrak{a}'$ is an ideal of $R[x_1,...,x_n]$ such that $\mathfrak{a}=\mathfrak{a}'\otimes_R F$, then we set $\lct_Y(\mathfrak{a}'):=\lct_Y(\mathfrak{a})$ for each subscheme $Y$ of $\AA_F^n$.  We refer the readers to \cite{Mustata2} for more information on log-canonical thresholds. 

Now, we recall another description of $\lct_P(\mathfrak{a})$ as follows. For each integer $l\geq 0$, we consider the truncation map $\rho_{\infty,l}:(\overline{F}[[t]])^n\to (\overline{F}[[t]]/(t^{l+1}))^n$. For each integer $m\geq 1$ and each $P\in Z(\mathfrak{a})(\overline{F})$, we set 
$${\rm Cont}_P^{\geq m}(\mathfrak{a})=\{x\in \overline{F}[[t]]^n\mid \rho_{\infty,0}(x)=P, \ord_t(\mathfrak{a}(x))\geq m\}$$
and 
$${\rm Cont}_{P,m}^{\geq m}(\mathfrak{a})=\rho_{\infty,m}({\rm Cont}_P^{\geq m}(\mathfrak{a}))\subset (\overline{F}[[t]]/(t^{m+1}))^n\cong \overline{F}^{(m+1)n},$$
where $\ord_t(\mathfrak{a}(x)):=\min_{g\in\mathfrak{a}}\ord_t(g(x))$. Then, we can view ${\rm Cont}_{P,m}^{\geq m}(\mathfrak{a})$ as a Zariski closed subset of $\AA_{\overline{F}}^{(m+1)n}$ for all integers $m\geq 1$. We have the following result.
\begin{prop}[Corollary 3.6, \cite{MustJAMS}]\label{jetlog}Let $F,\mathfrak{a}$ be as above. Let $P\in Z(\mathfrak{a})(\overline{F})$. With the above assumption that there is a log-resolution of $\mathfrak{a}$, one has
$$\lct_P(\mathfrak{a})\leq n-\frac{\dim({\rm Cont}_{P,m}^{\geq m}(\mathfrak{a}))}{m}$$
for all $m\geq 1$ and the equality holds for infinitely many $m\geq 1$.
\end{prop}

If $F\hookrightarrow F'$ is an injection for some field $F'$ then the log-resolution $\pi$ of $\mathfrak{a}$ induces a log-resolution $\pi_{F'}$ of $\mathfrak{a}\otimes_F F'$. Suppose that $F=K, F'=L$ for a number field $K$ and a local field $L$ over $\cO_K$ of characteristic zero. Suppose moreover that $\mathfrak{a}=\tilde{\mathfrak{a}}\otimes_{\cO_K}K$ for an ideal $\tilde{\mathfrak{a}}$ of $\cO_K[x_1,...,x_n]$. We sketch the definition of the reduction modulo $\cM_L$ of $\pi$ in \cite{Denefdegree}.  Since $\pi$ is projective, there is a positive integer $N$ such that $X'$ is a closed subscheme of $\PP_{X}^N$ and $\pi$ is the restriction of the projection $\PP_X^N\to X$ to $X'$.  If $Y$ is a closed subscheme of $\PP^N_{X_L}$, we denote by $\overline{Y}$ the reduction modulo $\varpi_L\cO_L$ of the Zariski closure $\tilde{Y}$ of $Y$ in $\PP_{\spec(\cO_L[x_1,...,x_n])}^N\supset\PP_{X_L}^N$, i.e., $\overline{Y}:=\tilde{Y}\otimes_{\cO_L}k_L$. We denote by $\tilde{\pi}$ the restriction of the projection  $\PP_{\spec(\cO_L[x_1,...,x_n])}^N\to \spec(\cO_L[x_1,...,x_n])$ to $\tilde{X}'_L$ and $\overline{\pi}:=\tilde{\pi}\otimes k_L: \overline{X'_L}\to \AA_{k_L}^n$. We say that the log-resolution $(X',\pi)$ of $\mathfrak{a}$ has \emph{good reduction modulo $\cM_L$} if $\overline{X'_L}$ and all $\overline{E_{iL}}$ are smooth, $\cup_{i\in \cT}\overline{E_{iL}}$ has only normal crossings, and the schemes $\overline{E_{iL}}$ and $\overline{E_{jL}}$ have no common component whenever $i\neq j$. If in addition, $N_i\notin\cM_L$ for all $i\in\cT$, we say that the log-resolution $(X',\pi)$ of $\mathfrak{a}$ has \emph{tame good reduction modulo $\cM_L$}.  Suppose that $(X',\pi)$ has good reduction modulo $\cM_L$. Then $(\overline{X'_L},\overline{\pi})$ is a log-resolution of $\tilde{\mathfrak{a}}\otimes_{\cO_K} k_L$ associated with the divisors $(\overline{E_{iL}})_{i\in \cT}$ and the numerical data $(\nu_i,N_i)_{i\in\cT}$. Thus, if $P\in\cO_L^n\cap Z(\mathfrak{a})(L)$ then 
\begin{equation}\label{gredu}
\lct_P(\mathfrak{a})=\min_{i\in \cT, \overline{P}\in\overline{\pi}(\overline{E_{iL}})(k_L)}\dfrac{\nu_i}{N_i}=\lct_{\overline{P}}(\tilde{\mathfrak{a}}\otimes_{\cO_K}k_L).
\end{equation}
Note that, there exists an integer $M$ depending only on $\pi$ such that if $L\in\cL_{K,M}$ then the log-resolution $(X',\pi)$ has tame good reduction mod $\cM_L$ (see, e.g.,  \cite[Theorem 2.4]{Denefdegree}).

\begin{rem}\label{expM1} Let $K$ be a number field and $f\in\cO_K[x_1,...,x_n]$ be a non-constant polynomial. Let $h$ be a log-resolution of the principal ideal $(f)$.  The claims of Proposition \ref{ex-zeta} hold for a local field $L$ over $\cO_K$ if there is $\tilde{L}\in\cL_{K,1}$ such that $k_L\simeq k_{\tilde{L}}$ and $h$ has tame good reduction modulo $\cM_{\tilde{L}}$. Thus the integer $M$ in Proposition \ref{ex-zeta} can be taken as the largest prime $p$ such that the log-resolution $h$ does not have tame reduction modulo $\mathfrak{p}K_\mathfrak{p}$ for some non-Archimedean place $\mathfrak{p}$ of $\cO_K$ over $p$, where $K_{\mathfrak{p}}$ is the completion of $K$ at $\mathfrak{p}$.
\end{rem}
\begin{rem}\label{genericsm}Suppose that $F$ is a field of characteristic zero. Let $U$ be an affine variety over $F$. Let $\mathfrak{b}\subset \cO(U)[x_1,...,x_n]$ be an ideal. Firstly, we suppose that $U$ is irreducible.  We denote by $F(U)$ the function field of $U$. Let $\eta:\cO(U)\to F(U)$ be the inclusion map.  Let $\pi:X\to \AA_{F(U)}^n$ be a log-resolution of $\mathfrak{b}\otimes_{\cO(U)}F(U)$. Let $F'$ be a field extension of $F$ and $P:\cO(U)\to F'$ be an $F'$-point of $U$. Then $P$ extends naturally to a ring homomorphism $\tilde{P}:F(U)\to F'$ such that $P=\tilde{P}\eta$. We denote by $Id_n:F[x_1,...,x_n]\to F[x_1,...,x_n]$ the identity map. Let $Y_P$ be the fiber of $U\times_{\spec(F)} \AA_F^n$ over $P$ then $Y_P$ can be identified with $\AA_{F'}^n$ and the inclusion morphism $Y_P\to U\times_{\spec(F)} \AA_F^n$ corresponds to the ring homomorphism  $P\otimes_FId_n$. We set $\mathfrak{b}_P=(P\otimes_FId_n)(\mathfrak{b})F'[x_1,...,x_n]$.  The homomorphism $\tilde{P}\otimes_F Id_n: F(U)[x_1,...,x_n]\to F'[x_1,...,x_n]$ induces a morphism   $\beta_P:Y_P\to \AA_{F(U)}^n$. Let $\pi_P: X_P\to Y_P$ be the base change of $\pi$ via $\beta_P$. By using the generic smoothness as in \cite[Property 1.23]{Mustata2}, we can find a non-empty open affine subset $V$ of $U$ such that  $\pi_P$ is a log-resolution of $\mathfrak{b}_P$ for all field extension $F'$ of $F$ and all $F'$-point $P$ of $V$. We will say that $\pi$ is a common log-resolution of $\mathfrak{b}$ over $V$.  By using the induction on the dimension of $U$, we can write $U$ as a finite disjoint union of locally closed subsets $U_i$  of $U$ such that  there is a common log-resolution of $\mathfrak{b}|_{U_i}$ for each $i$.  This fact also holds for arbitrary affine variety $U$ over $F$ since we can write $U$ as a finite disjoint union of irreducible  affine variety over $F$. 

Let $d,e$ be positive integers. We consider the parameter space $U$ of all tuples of $e$ polynomials in $n$ variables over $F$ of degree at most $d$. Then $U$ is an affine space over $F$ and we can form naturally an ideal $\mathfrak{b}$ of $\cO(U)[x_1,...,x_n]$ such that for each ideal $\mathfrak{a}$ of $F[x_1,...,x_n]$ that can be generated by $e$ polynomials of degree at most $d$, there is an $F$-point $P$ of $U$ with $\mathfrak{b}_P=\mathfrak{a}$. We do the above process to write $U$ as a finite disjoint union of locally closed subsets $U_i$  of $U$ such that  there is a common log-resolution $\pi_i$ of $\mathfrak{b}|_{U_i}$ for each $i$. This process depends only on $n,d,e$. In particular, for each ideal $\mathfrak{a}$ of $F[x_1,...,x_n]$ that can be generated by $e$ polynomials of degree at most $d$, there is an index $i$ such that $\pi_i$ induces a log-resolution of $\mathfrak{a}$. Thus, there is a number $N_{d,e,n}$ depending only on $d,e,n$ such that for each ideal $\mathfrak{a}$ of $F[x_1,...,x_n]$ that can be generated by $e$ polynomials of degree at most $d$, there is a log-resolution $(X',h)$ of $\mathfrak{a}$ such that the complexity of $(X',h)$ is bounded by $N_{d,e,n}$.
\end{rem}

\subsection{Transfer principle for  exponential sums modulo $p^m$}\label{setup3}
 Let $K$ be a number field and $k$ be a field over $\cO_K$, i.e., $k$ is equipped with a structure of  $\cO_K$-algebra $\varphi:\cO_K\to k$. If $Y$ is a subscheme of $\AA_{\cO_K}^n$, we write  $Y_k$ instead of $Y\otimes_{\Spec(\cO_K)}\Spec(k)$, as usual. Let $f\in\cO_K[x_1,...,x_n]$. From now on, we also write $f$ for the polynomial $\varphi(f)\in k[x_1,...,x_n]$. For each $x=(x_1,...,x_n)\in k[[t]]^n$ one can write
$$x_j=\sum_{i\geq 0}x_{ij}t^i.$$
For $m$ a non-negative integer and $x\in k[[t]]^n$, we set $$x^{(m)}:=(x_{ij})_{0\leq i\leq m, 1\leq j\leq n}\in k^{(m+1)n}.$$ By using these notations, one has
$$f(x)=\sum_{i\geq 0}f_i(x^{(i)})t^i,$$
where $f_i$ is a polynomial in $k[x^{(i)}]$ for all $i\geq 0$. If $Y$ is a subscheme of $\AA_k^n$, we identify $Y(k)+ (tk[[t]]/(t^{m+1}))^n\subset (k[[t]]/(t^{m+1}))^n$ with the set of $k$-points of the $k$-scheme $Y^{(m)}:=(Y\otimes_{\spec(k)}\AA_{k}^{mn})$. Then we can view $f_m$ as a regular function on $Y^{(m)}$. If $Y=\{P\}$ for a point $P\in k^n$, we write $P^{(m)}$ instead of $\{P\}^{(m)}$.

In this paper, we will need some important results from \cite{NguyenVeys}.
\begin{prop}\cite[Proposition 4.3 and 5.4]{NguyenVeys}\label{transfer1}Let $Y$ be a subscheme of $\AA_{\cO_K}^n$. With the notations as above, if $L\in\cL'_{K,m+1}$ for some $m\geq 2$ and $\psi$ is any additive character of $L$ of conductor $m$ then there exist elements $a\in k_L^{\times}$ and $b\in k_L$ depending only on $\psi$ and satisfying
$$E_{f}(Y,\psi)=q_L^{-mn}\sum_{x^{(m-1)}\in Y_{k_L}^{(m-1)}(k_L)}\exp_{p_L}(\Tr_{k_L/\FF_{p_L}}(bf_0(x^{(0)})+af_{m-1}(x^{(m-1)})))$$
where $\exp_{p_L}:\FF_{p_L}\to \CC^{\times}$, $z\mapsto \exp(\frac{2\pi iz}{p_L})$ is the standard additive character of $\FF_{p_L}$.
 
Moreover, suppose that $f(Y(\CC))=0$ and $M$ is as in Remark \ref{expM1}. Let $L, L'\in\tilde{\cL}_{K,M}$ such that $k_L\simeq k_{L'}$. If $\psi$ is an additive character of $L$ of conductor $m_\psi\geq 2$ then there exists an additive character $\psi'$ of $L'$ of conductor $m_{\psi'}=m_\psi$ such that
$$E_{f}(Y,\psi)=E_f(Y,\psi').$$
\end{prop}
\begin{rem}In the proof of \cite[Proposition 5.4]{NguyenVeys}, it is not clear that one can take $M$ as in Remark \ref{expM1}. However, we can use Proposition \ref{Fourier} and \ref{ex-zeta} to see this fact.
\end{rem}
\begin{prop}\cite[Corollary 5.5]{NguyenVeys}\label{transfer2} Let $(Y_i)_{i\in J}$ be a family of subschemes of $\AA_{\cO_K}^n$ and $(\sigma_i)_{i\in J}$ be a family of  positive real numbers. With the notations as above, suppose that for each $m> 1$, there exists an integer $M_m$ and a positive constant $C_m$ such that for all $i\in J$, all local fields $L\in \cL'_{K,M_m}$ and all additive characters $\psi$ of $L$ of conductor $m$ we have
$$|E_f(Y_i,\psi)|\leq C_m q_L^{-m\sigma_i}.$$
Then there exists an integer $M$ and a positive constant $C$ such that for all $i\in J$, all local fields $L\in\tilde{\cL}_{K,M}$ and all additive characters $\psi$ of $L$ of conductor $m_\psi>1$, we have
$$|E_{f}(Y_i,\psi)|\leq Cm_{\psi}^{n-1}q_L^{-m_\psi\sigma_i}.$$
\end{prop}
\begin{rem}\label{exp1}
As in the proof of \cite[Claim 3.2.7]{Saskia-Kien}, the constant $M$ in Proposition \ref{transfer2} depends only on $f$ and the set $\{M_i|i\in I\}$ for a finite set $I\subset J$ depending only on $f$. In particular, we can take $I$ depending only on the numerical data associated with a log-resolution of $f$. By using log-resolutions of families as in \cite[Property 1.23]{Mustata2}, there exists a finite set $A$ of numerical data such that if $f$ is a polynomial of degree $d$ in $n$ variables over $\CC$, we can find a log-resolution of $f$ whose numerical data belongs to $A$. Therefore we can take $I$ depending only on $n,d$.

On the other hand, again by \cite[Claim 3.2.7]{Saskia-Kien}, we can take $C$ depending only on the set $\{C_i|i\in I\}$.
\end{rem}
\begin{rem}\label{algebraic}Proposition \ref{transfer2} still holds for any family of subschemes of $\AA^n_{\cO_K^{\rm alg}}$ if we use the convention that $E_{f}(Y,\psi)=0$ for all additive characters $\psi$ of $L$ whenever $Y$ is not defined over $\cO_L$.
\end{rem}

\section{Uniform bound for exponential sums modulo $p^m$ when $p$ is  large enough}\label{bound}
In this section, we fix a number field $K$ and a non-constant polynomial $f\in\cO_K[x_1,...,x_n]$. We set $s:=s(f)$ the dimension of the critical locus of the top degree part of $f$ as mentioned in the introduction. At the beginning of this section, we study the exponential sums at points $E_{f}(P,\psi)$ where $\psi$ is any additive character of conductor $m_\psi>1$ on a local field $L$ over $\cO_K$ of large enough residue field characteristic and $P\in k_L^n$. After that, we will glue information on the exponential sums at points to obtain bound for the exponential sums associated with arbitrary subscheme $Y$ of $\AA_{\cO_K}^n$. 

First of all, for each $1\leq j\leq n$, we set $F_j=\frac{\partial f}{\partial x_j}$. 
Let $k$ be a field over $\cO_K$. With the notations of Subsection \ref{setup3}, for each $P=(a_1,...,a_n)\in k^n$, $1\leq j\leq n$ and $m\geq 1$, we set
$$f_{P,m}(\tilde{x}^{(m)}):=f_m(x_P^{(m)})$$
and 
$$(F_j)_{P,m}(\tilde{x}^{(m)}):=(F_j)_{m}(x_P^{(m)}),$$
where $\tilde{x}^{(m)}=(x_{ij})_{1\leq i\leq m, 1\leq j\leq n}$ and $x_P^{(m)}=(x_{ij})_{0\leq i\leq m, 1\leq j\leq n}$ with $x_{0j}=a_j$ for all $1\leq j\leq n$. We use the convention that $f_{P,m}=0=(F_j)_{P,m}$ if $m<0$, $f_{P,0}=f(P)$ and $(F_j)_{P,0}=F_j(P)$. Moreover, we identify $P^{(m)}(k)$ with $k^{mn}$. From now on, we suppose that the degree $d$ of $f$ is at least $2$.
\begin{lem}\label{locussing} Let $k$ be a field over $\cO_K$, $P\in k^n$ be a critical point of $f$ and $m\geq 1$. Then $f_{P,m}$ is a weighted homogeneous polynomial and its critical locus is 
$$\Sing(f_{P,m})=\cap_{1\leq j\leq n, 0\leq i\leq m-1}Z((F_{j})_{P,i}).$$
\end{lem}
\begin{proof}It is sufficient to work with the case that $P=(0,...,0)$ and $f(P)=F_1(P)=...=F_n(P)=0$.
It is clear that $f_{P,m}$ is weighted homogeneous for all $m\geq 1$. If $x=(x_1(t),...,x_n(t))\in tk[[t]]^n$ with
$$x_j(t)=x_{1j}t+x_{2j}t^2+\ldots$$
for $1\leq j\leq n$, then we consider the formal function
$$F(t,\tilde{x}^{(\infty)}):=f(x)=\sum_{i\geq 1}f_{P,i}(\tilde{x}^{(i)})t^i$$
where $\tilde{x}^{(\infty)}=(x_{ij})_{1\leq i, 1\leq j\leq n}$. If $1\leq \ell$ and  $1\leq j\leq n$  then
$$\frac{\partial}{\partial x_{\ell j}}(F(t,\tilde{x}^{(\infty)}))=\frac{\partial}{\partial x_{\ell j}}(\sum_{i\geq 1}f_{P,i}(\tilde{x}^{(i)})t^i)=\sum_{i\geq 1}\frac{\partial f_{P,i}}{\partial x_{\ell j}}(\tilde{x}^{(i)})t^i$$
and
\begin{align*}
\frac{\partial}{\partial x_{\ell j}}(F(t,\tilde{x}^{(\infty)}))&=\frac{\partial}{\partial x_{\ell j}}(f(x_1(t),...,x_n(t)))\\
&=\frac{\partial f}{\partial x_j}(x_1(t),...,x_n(t))\frac{\partial}{\partial x_{\ell j}}(x_j(t))\\
&=F_j(x_1(t),...,x_n(t))t^\ell\\
&=\sum_{i\geq 1}(F_{j})_{P,i}(\tilde{x}^{(i)})t^{i+\ell}
\end{align*}
Hence,
\begin{equation}\label{Derivative1}
\frac{\partial f_{P,i}}{\partial x_{\ell j}}(\tilde{x}^{(i)})=(F_{j})_{P,i-\ell}(\tilde{x}^{(i-\ell)})
\end{equation}
for all $\ell, j$ with the convention that $(F_{j})_{P,i}=0$ if $i\leq 0$. Thus, the assertion follows from (\ref{Derivative1}).
\end{proof}
\begin{defn}\label{deflog}Let $J_f\subset\cO_K[x_1,...,x_n]$ be the Jacobian ideal of $f$ over $\cO_K$ as mentioned in Subsection \ref{setup1}. If $F$ is a field containing $K$, we write $J_f$ instead of $J_f\otimes_{\cO_K}F$. For each local field $L$ over $\cO_K$ of characteristic zero and  each $P\in k_L^n$ we set
$$\sigma_{P}(L,J_f)=\min_{Q\in \cO_L^n,\overline{Q}=P}\lct_Q(J_{f}),$$
where $\lct_Q(J_{f})$ is the log canonical threshold of $J_{f}$ at $Q$ defined in Subsection \ref{setup2}.
\end{defn}
\begin{rem}\label{reduction}For each $L\in\cL_{K,1}$ and $P\in k_L^n$, we have $\sigma_{P}(L,J_f)=+\infty$ if there is no critical point of $f$ in the set $\{x\in\cO_L^n|\overline{x}=P\}$. Let $h:X'\to \AA_K^n$ be a log-resolution of $J_{f}$. There is an integer $M$ such that $h$ has good reduction modulo $\cM_L$ provided that $p_L>M$. Thus, if $p_L>M$ and $P\in k_L^n$ is a critical point of $f$ then there exists $Q$ in $\cO_L^n$ such that $\overline{Q}=P$ and $Q$ is a critical point of $f$. 
\end{rem}
\begin{lem}\label{boundsing}There exists an integer $M$ depending only on $f$ such that 
\begin{equation}\label{localsing}
\dim(\Sing(f_{P,m-1}))\leq mn-(m-1)\sigma_{P}(L,J_f)
\end{equation}
for all $m\geq 2$, all local fields $L\in\cL_{K,M}$ and all critical points $P\in k_L^n$ of $f$.
\end{lem}
\begin{proof}Let $h,M$ be as in Remark \ref{reduction}. 
If $L\in\cL_{K,M}$, then $h$ induces a log-resolution $h_L$ of $J_{f}$ and $h_L$ induces a log-resolution $h_{k_L}$ of $J_f\otimes_{\cO_K}k_L$. Thus, it follows from Proposition \ref{jetlog}, Lemma \ref{locussing}, Remark \ref{reduction} and  (\ref{gredu}) that 
\begin{equation*}\label{log-canonical1}
\dim(\Sing(f_{P,m-1}))\leq mn-(m-1)\lct_P(J_f\otimes_{\cO_L}k_L)=mn-(m-1)\sigma_{P}(L,J_f)
\end{equation*}   
or all $m\geq 2$, all local fields $L\in\cL_{K,M}$ and all critical points $P\in k_L^n$ of $f$.
\end{proof}

\begin{prop}\label{sumsfinite1}For each $m\geq 2$, there exists a constant $M_m$ depending only on $m,f$ and a constant $B_m$ depending only on $m, n, d$  such that for all finite fields $k$ over $\cO_K$ of characteristic at least $M_m$ with structure homomorphism $\varphi:\cO_K\to k$, all $a\in k^{\times}$ and all critical points $P\in k^n$ of $f$ we have
$$||k|^{-mn}\sum_{\tilde{x}^{(m-1)}\in k^{(m-1)n}}\exp_k(af_{P,m-1}(\tilde{x}^{(m-1)}))|\leq B_m|k|^{\frac{-m\sigma_{P}(L,J_f)+\sigma_{P}(L,J_f)-n}{2}}$$
if $L$ is a local field over $\cO_K$ of characteristic zero such that $k_L=k$ and the structure homomorphism $\varphi:\cO_K\to k$ is the reduction modulo $\cM_L$.
\end{prop}
\begin{proof}
The assertion follows by combining (\ref{localsing}) with \cite[Theorem 4]{Katz} and the proof of \cite[Theorem 7.4]{CDenSperlocal}.
\end{proof}
\begin{rem}\label{reduction1}Let $L\in\cL_{K,1}$. With a log-resolution $h$ as in Remark \ref{reduction}, it follows from Lemma \ref{boundsing} and \cite[Theorem 4]{Katz} that Proposition \ref{sumsfinite1} holds  for $k_L$ if $h$ has good reduction modulo $\cM_L$ and $p_L>m$. In particular, if $M$ is as in Remark \ref{reduction} then we can take $M_m=\max\{m,M\}$.
\end{rem}
\begin{prop}\label{sumsfinite2}For each $m\geq 2$, there exists a constant $M_m$ depending only on $m,f$ and a constant $B_m$ depending only on $m, n, d$ such that for all finite fields $k$ over $\cO_K$ of characteristic at least $M_m$ with structure homomorphism $\varphi:\cO_K\to k$, all $a\in k^{\times}$ and all critical points $P\in k^n$ of $f$ we have
$$||k|^{-mn}\sum_{\tilde{x}^{(m-1)}\in k^{(m-1)n}}\exp_k(af_{P,m-1}(\tilde{x}^{(m-1)}))|\leq B_m|k|^{\frac{-m\sigma_{P}(L,J_f)}{2}}$$
if $L$ is a local field over $\cO_K$ of characteristic zero such that $k_L=k$ and the structure homomorphism $\varphi:\cO_K\to k$ is the reduction modulo $\cM_L$.
\end{prop}
\begin{proof}
 If $L\in\cL_{K,1}$ and $Q\in L^n$ is a critical point of $f$ then  it follows from Proposition \ref{jetlog} that $\lct_Q(J_{f})\leq n$. 
Thus, we can combine  Definition \ref{deflog}, Remark \ref{reduction} and Proposition \ref{sumsfinite1} to imply the assertion.
\end{proof}

\begin{prop}\label{uniformlocal1}There exists a constant $M$ depending only on $f$ and a constant $C$ depending only on $n,d$ such that for all local fields $L\in\tilde{\cL}_{K,M}$, all critical points $P\in k_L^n$ of $f$ and all additive characters $\psi$ of $L$ of conductor $m_\psi\geq 2$ we have
$$|E_{f}(P,\psi)|\leq Cm_\psi^{n-1}q_L^{\frac{-m_\psi\sigma_{P}(\tilde{L},J_f)}{2}},$$
where $\tilde{L}$ is a local field over $\cO_K$ of characteristic zero with $k_{\tilde{L}}\simeq k_L$.
\end{prop}
\begin{proof}
It follows from  Propositions \ref{transfer1}, \ref{transfer2}, \ref{sumsfinite2} and  Remarks  \ref{exp1}, \ref{algebraic}.
\end{proof}
\begin{cor}\label{uniformlocal2}
There exists a constant $M$ depending only on $f$ and a constant $C$ depending only on $n,d$ such that for  all local fields $L\in\tilde{\cL}_{K,M}$, all additive characters $\psi$ of $L$ of conductor $m_\psi\geq 2$ and all critical points $P\in k_L^n$ of $f$ we have
$$|E_{f}(P,\psi)|\leq Cm_\psi^{n-1}q_L^{\frac{-m_\psi(n-s)}{2(d-1)}}.$$
\end{cor}
\begin{proof}As in Section \ref{sec:intro}, we denote by $f_d$ the homogeneous part of degree $d$ of $f$. Let $L$ be a local field over $\cO_K$ of characteristic $0$. We denote by $J_{f_d}\subset L[x_1,...,x_n]$ the Jacobian ideal of $f_d$ over $L$. By Proposition \ref{uniformlocal1}, Definition \ref{deflog} and Remark \ref{reduction}, it remains to show that  
\begin{equation}\label{lowerbound}
\lct_Q(J_{f})\geq \frac{n-s}{d-1}
\end{equation}
if $Q\in \cO_L^n$ is a critical point of $f$.
First of all, we can suppose that $Q=0$ by a suitable change of coordinates. We can take a hyperplane $H\simeq\AA_L^{n-1}$ passing through the origin $0$ such that $g:=f_d|_{H}\neq 0$ and $\dim(\Sing(g))=s-1$. We denote by $J_g\subset\cO_H(H)\simeq L[y_1,...,y_{n-1}]$ the Jacobian ideal of $g$ over $L$. By \cite[Property 1.17]{Mustata2}, we have $\lct_0(J_{f})\geq \lct_0(J_{f}|_H)$. Moreover, it is easy to see that $J_{g}\subset J_{f}|_H$. Thus,  we can use \cite[Property 1.12]{Mustata2} to deduce that $\lct_0(J_{f})\geq \lct_0(J_{f}|_H)\geq \lct_0(J_g)$. Therefore, by induction, it is sufficient to prove (\ref{lowerbound}) in the case that $s=0$. Let us check (\ref{lowerbound}) when $s=0$. Since the log canonical threshold is stable under changing coordinates, for each $a\neq 0$ we have $\lct_0(a^{d-1}J_{f}(x/a))=\lct_{0}(J_{f}(x/a))=\lct_0(J_{f}(x))$. By the semicontinuity of the log canonical threshold in \cite[Property 1.24]{Mustata2}, we have $\lct_0(a^{d-1}J_{f}(x/a))\geq \lct_0(J_{f_d})$. Note that $\dim(\Sing(f_d))=0$ and $f_d$ is homogeneous. Thus, for each $x\in t\overline{L}[[t]]^n$,  $\ord_t(J_{f_d}(x))=\min_{1\leq i\leq n}\ord_t(\frac{\partial f_d}{\partial x_i}(x))\geq m$ iff $\ord_t(x)\geq \frac{m}{d-1}$.  Therefore, Proposition \ref{jetlog} implies $\lct_0(J_{f_d})=\frac{n}{d-1}$, as desired.
\end{proof}
\begin{cor}\label{uniformlocal3}There exists a constant $M$ depending only on $f$ and a constant $C$ depending only on $n,d$ such that for all local fields $L\in\tilde{\cL}_{K,M}$, all additive characters $\psi$ of $L$ of conductor $m_\psi\geq 2$ and all critical points  $P\in k_L^n$ of $f$ we have
$$|E_{f}(P,\psi)|\leq Cm_\psi^{n-1}p^{-m_\psi\tilde{\sigma}_{P,\tilde{L}}},$$
with $\tilde{L}$ as in Proposition \ref{uniformlocal1} and with
\begin{equation}\label{localinva}
\tilde{\sigma}_{P,\tilde{L}}=\min_{Q\in C_f(\tilde{L}), P=\overline{Q}}\frac{n-s_{Q}}{2(d_{Q}-1)}.
\end{equation}
Where if $Q\in C_f(\tilde{L})$ then $d_Q>1$ and $s_Q$ are natural numbers such that  one can write $f(x+Q)=f(Q)+f_{d_Q}(x)+...+f_d(x)$ with $f_{d_Q}\neq 0$, $f_i$ is homogeneous for all $i$ and $s_Q=\dim(\Sing(f_{d_Q}))$.
\end{cor}
\begin{proof}As in the proof of Corollary \ref{uniformlocal2}, it suffices to show that $\lct_0(J_{f})=\frac{n}{d_{0}-1}$ if $0$ is a critical point of $f$ and $s_0=0$. Since $\dim(\Sing(f_{d_0}))=0$, if $x\in t\overline{L}[[t]]^n$ then $\ord_t(J_{f}(x))\geq m$ iff $\ord_t(x)\geq \frac{m}{d_0-1}$. Thus, we can use Proposition \ref{jetlog} to have
$$\lct_0(J_{f})=\frac{n}{d_{0}-1}.$$
\end{proof}
We need the following lemma to go from exponential sums at points to exponential sums associated with schemes.
\begin{lem}
\label{padic}Let $K$ be a number field. Let $g$ be a non-constant polynomial in $\cO_K[x_1,...,x_n]$. Let $H_1$ and $H_2$ be linear subspaces of $\AA_{\cO_K}^n$, of codimension $r$, resp.~codimension $n-r$, and so that $H_1(\CC)+H_2(\CC)=\CC^n$.
Then there is $M>0$ and a nonzero polynomial $u$ in $r$ variables defined over $\cO_K$ such that for any $b\in H_2(\cO_K^{\rm alg})$, and $Q\in b+H_1(\cO_K^{\rm alg})$, any local field $L$ over
 $\cO_K[b,Q]$ with residue field characteristic at least $M$, and any multiplicative character $\chi$ on $\cO_L^{\times}$ with $c(\chi)=1$, if one has $\ord_L (u(b)) = 0$,
then
\begin{equation}\label{eq:H1H2}
\cZ_{\chi}(g,L,Q,s) = q_L^{-r} \cZ_{\chi}(g_{|b+H_1},L,Q,s).
\end{equation}
Furthermore, the degree of $u$ can be bounded in terms of  $n$ and $d=\deg(g)$.
\end{lem}
\begin{proof}
This is a variant of \cite[Lemma 4.1]{Cigumodp}, with a similar proof based on Denef-Pas cell decomposition \cite{Pas} and the study of parameter integrals of \cite{CLoes}, which works uniformly for big enough $p_L$ since $c(\chi)=1$. To see our claim on the degree of $u$, we can proceed similarly to \cite[Lemma 4.1]{Cigumodp} but with more parameters come from coefficients defining $f,H_1,H_2$.
\end{proof}
\begin{prop}\label{generalZ}Let $Y$ be an $\cO_K$-subscheme of $\AA_{\cO_K}^n$ of finite type whose complexity is bounded by $(R,D)$. Moreover, suppose that $f(Y(\CC))=0$. There exists a constant $M$ depending only on $Y,f$ and a constant $C$ depending only on $n, d, R, D$ such that for all local fields $L\in\tilde{\cL}_{K,M}$ and all additive characters $\psi$ of $L$ of conductor $m_\psi\geq 2$, we have
$$|E_{f}(Y,\psi)|\leq Cm_\psi^{n-1}p^{\frac{-m_\psi(n-s)}{2(d-1)}}.$$
\end{prop}
\begin{proof}
 We follow a strategy from \cite{Cigumodp}, and use induction on the dimension of the intersection of $Y_K$ with the singular locus $\Sing(f)$ of $f$ over $K$. 
If the intersection of $Y_K$ with $\Sing(f)$ has dimension zero, there is nothing to prove by Proposition \ref{ex-zeta} and Corollary \ref{uniformlocal2}. 
Now let $Y$ such that $f(Y(\CC))=0$. We may suppose that $Y_K$ is contained in the singular locus of $f$  since the exponential sum vanishes away from the singular locus as in Proposition \ref{ex-zeta}. Thus we can suppose that the dimension of $Y_K$ is $r\leq \dim(\Sing(f))$. Note the easy fact that $\dim(\Sing(f))\leq \dim(\Sing(f_d))=s$ and thus $r\leq s$. There exists a closed subscheme $Y''$ of $Y$ of finite type over $\cO_K$ and a linear subspace $H_1$ of $\AA_{\cO_K}^n$ of codimension $r$, such that the  codimension  of $Y''_K$ in $Y_K$ is at least $1$, $f_d|_{H_{1K}}\neq 0$, $\dim(\Sing(f_d|_{H_{1K}}))=s-r$ and the set
of $\CC$-rational points on $(Y\setminus Y'')(\CC)\cap (a+H_1(\CC))$  is finite for any $a\in Y(\CC)\setminus Y''(\CC)$. Moreover, there is a linear subspace $H_2$ of $\AA_{\cO_K}^n$ of codimension $n-r$ such that $H_1(\CC)+H_2(\CC)=\CC^n$, the sum being direct.

By Lemma \ref{padic}, there is a closed subscheme $Y'$ of $Y$ of finite type over $\cO_K$, such that the relation (\ref{eq:H1H2}) between $\cZ_{\chi}(f,L,Q,s)$ and $\cZ_{\chi}(f_{|Q+H_1},L,Q,s)$ holds for any $Q$ and any character $\chi$ of $\cO_L^{\times}$  as soon as $c(\chi)=1$, ${\overline {Q}}\in Y(k_L)\setminus Y'(k_L)$ and $p_L>M_0$. Let $u$ be as in Lemma \ref{padic} with respect to $f, H_1,H_2$. By the condition that $(Y\setminus Y'')(\CC)\cap (a+H_1(\CC))$ is finite for any $a\in Y(\CC)\setminus Y''(\CC)$, therefore the projection $\rho$ from $Y'_K\setminus Y''_K$ to $H_{2K}$ along the direction $H_{1K}$ is quasi-finite.  By Lemma \ref{padic} and the property of $Y'$, we must have $\rho(Y'_K\setminus Y''_K)\subset Z(u)$. Here, $Z(u)\subset \AA_K^r$ is the zero locus of $u$. Thus, we have $\dim(Y'_K)\leq \max(\dim(Y''_K),\dim(Z(u)))\leq r-1<\dim(Y_K)$. Put $Y_1=Y\setminus (Y'\cup Y'')$. If $p_L>M_0$ then by (\ref{eq:H1H2}) we have
\begin{align*}
\cZ_{\chi}(f,L,Y_1,s) &= \cZ_{\chi}(f,L,Y,s) - \cZ_{\chi}(f,L,Y'\cup Y'',s)\\
&=q_L^{-\dim(Y_{1K})}\sum_{\overline{Q}\in Y_1(k_L) }\cZ_\chi(f_{|Q+H_1},L,Q,s).
\end{align*}

We can shrink $Y_1$ such that $\dim(Y_K\setminus Y_{1K})<\dim(Y_K)$, $Y_1\sqcup (Y\setminus Y_1)$ is still of finite type over $\cO_K$ and there exists a common log-resolution $h_1$ defined over $K$ of singularities of the family  $(f|_{Q+H_{1K}})_{Q\in Y_{1K}}$ as mentioned in Remark \ref{genericsm}. By the definition of good reduction in Subsection \ref{setup2}, there exist regular functions $g_1,..,g_e$ on $Y_{1K}$ defined over $K$ such that if $L$ is a local field over $\cO_K$ of characteristic zero and $h_1,g_1,...,g_e$ are defined over $\cO_L$ then the specialization $h_{1Q}$ of $h_1$ at $Q\in Y_1(L)\cap \cO_L^n$ has good reduction modulo $\cM_L$ iff $\overline{g}_i(\overline{Q})\neq 0$ for all $1\leq i\leq e$. For each point $Q\in Y_1(\overline{K})$ and $L\in\cL_{K,1}$ such that $Q\in \cO_L^n$ then $h_{1Q}$ will have good reduction modulo $\cM_L$ if $p_L$ is large enough. Thus we see that $g_1,...,g_e$ must be invertible functions on $Y_1(\CC)$. There exists an integer $M_1>M_0$ such that $h_1,g_1,...,g_e$ are defined over $\cO_L$ if $L\in\cL_{K,M_1}$. By logical compactness (see, e.g., \cite[Corollary 2.2.10]{Marker}), we can enlarge $M_1$ such that $\overline{g}_1,...,\overline{g}_e$ are invertible functions on $Y_1(\overline{\FF}_p)$ if $p>M_1$. Therefore,  if $L\in\cL_{K,M_1}$ and $Q\in Y_1(L)\cap \cO_L^n$, then  $h_{1Q}$ has good reduction modulo $\cM_L$ if $\overline{Q}\in Y_1(k_L)$. By enlarging $M_1$ if needed, $h_{1Q}$ has tame good reduction modulo $\cM_L$ if $L\in\cL_{K,M_1}$ and $\overline{Q}\in Y_1(k_L)$. Thus Propositions \ref{Fourier},  \ref{ex-zeta}, \ref{transfer1} and Remark \ref{expM1} yield
$$E_{f}(Y_1,\psi)=q_L^{-\dim(Y_{1K})}\sum_{\overline{Q}\in Y_1(k_L)}E_{f_{|Q+H_1}}(Q,\psi)$$
for all local fields $L\in\tilde{\cL}_{K,M_1}$ and all additive characters $\psi$ of $L$ of conductor $m_\psi\geq 2$. 

By Remark \ref{genericsm} and shrinking $Y_1$ again,  we may suppose that there exists a common log-resolution $h_2$ of singularities of the family  $(J_{f|_{Q+H_{1K}}})_{Q\in Y_{1K}}$.  By the same reason as above, we also have an integer $M_2>M_1$ such that if $L\in\cL_{K,M_2}$, $Q\in Y_1(L)\cap \cO_L^n$ and $\overline{Q}\in Y_1(k_L)$ then the specialization $h_{2Q}$ of $h_2$ at $Q$ will have good reduction modulo $\cM_L$. Thus the discussion in Remark \ref{reduction1} implies a uniform version of Proposition \ref{sumsfinite2} for the family of polynomials $(f|_{Q+H_1})_{\overline{Q}\in Y_1(k_L)}$ if $p_L>M_2$. By using generic flatness and enlarging $M_2$ if needed, we can suppose that $\dim(Y_{1K})=\dim(Y_{1k_L})$ if $L\in\cL_{K,M_2}$. By combining this uniform version with  Proposition \ref{transfer1} and the proof of Corollary \ref{uniformlocal2}, we deduce that for each $m\geq 2$ there exists a constant $M_m$ depending only on $f,Y,m$ and a constant $B_m$ depending only on $m,n,d$ such that for all local  fields $L\in\tilde{\cL}_{K,M_m}$, all additive characters $\psi$ of $L$ of conductor $m$ and all $Q$ with $\overline{Q}\in Y_1(k_L)$, we have
$$|E_{f_{|Q+H_1}}(Q,\psi)|\leq B_mq_L^{\frac{-m(n-s)}{2(d-1)}}.$$
By Lemma \ref{padic} and Remark \ref{genericsm}, the complexity of $Y'\sqcup Y''\sqcup Y_1$ can be bounded in terms of the complexity of $f,Y$. Thus, by the Lang-Weil estimate (see, e.g., \cite[Proposition 3.2]{CMN} and \cite{Katz-Bet}), we have
\begin{align*}
|E_{f}(Y_1,\psi)|&=q_L^{-\dim(Y_{1K})}\left|\sum_{\overline{Q}\in Y_1(k_L)}E_{f_{|Q+H_1}}(Q,\psi)\right|\\
&\leq |Y_1(k_L)|q_L^{-\dim(Y_{1K})}B_m q_L^{\frac{-m(n-s)}{2(d-1)}}\\
&\leq c_{Y_1}B_m q_L^{\frac{-m(n-s)}{2(d-1)}}.
\end{align*}
for all local fields $L\in\tilde{\cL}_{K,M_m}$, all additive characters $\psi$ of $L$ of conductor $m\geq 2$ and for a constant $c_{Y_1}$ depending only on the complexity of $Y_1$. By Proposition \ref{transfer2} and Remark \ref{exp1}, there exists a constant $M$ depending only on $Y,f$  and a constant $C$ depending only on $n, d, R, D$ such that for all local fields $L\in\tilde{\cL}_{K,M}$ and all additive characters $\psi$ of $L$ of conductor $m_\psi\geq 2$, we have
$$|E_{f}(Y_1,\psi)|\leq Cm_\psi^{n-1}q_L^{\frac{-m_\psi(n-s)}{2(d-1)}}.$$
This reduces the problem to $Y'\cup Y''$ instead of $Y$. Hence, our claim follows by induction on the dimension of $Y_K\cap\Sing(f)$.
\end{proof}
In fact, we can remove the term  $m_\psi^{n-1}$ in the assertion of Proposition \ref{generalZ} as follows. 
\begin{prop}\label{generalZ1}Let $Y$ be a subscheme of $\AA_{\cO_K}^n$ whose complexity is bounded by $(R,D)$. Suppose that $f(Y(\CC))=0$. There exists an integer $M_3$ depending only on $f,Y$ and a constant $C_3$ depending only on $n,d, R, D$ such that for all local fields $L\in\tilde{\cL}_{K,M_3}$ and all additive characters  $\psi$ of $L$ of conductor $m_\psi\geq 2$, we have
$$|E_{f}(Y,\psi)|\leq C_3q_L^{\frac{-m_\psi(n-s)}{2(d-1)}}.$$
\end{prop}
\begin{proof}We fix a log-resolution $h:X\to \AA_K^n$ of $(f)+J_f^2$ over $K$. By Remark \ref{genericsm}, we can choose $(X,h)$ such that the complexity of $(X,h)$ is bounded in terms of $n,d$. As mentioned in Subsection \ref{setup2}, there is an integer $M_1$ such that $h$ has good reduction modulo $\cM_L$ for all $L\in \cL_{K,M_1}$. By using the last inequality in the proof of \cite[Proposition 6.4]{CKMu} and enlarging $M_1$ if needed, we can suppose that 
\begin{equation}\label{Jff0}
|E_{f}(Y,\psi)|\leq \int_{\{x\in \cO_L^n, \overline{x}\in Y(k_L), \ord_L(((f)+J_f^2)(x))\geq m_\psi-1\}}|dx|
\end{equation}
for all $L\in \cL_{K,M_1}$, all additive characters $\psi$ of $L$ of conductor $m_\psi\geq 2$. Note that we can take $M_1$ depending only on $f$ by combining the arguments of \cite[Proposition 6.4]{CKMu} and \cite[Proposition 5.4]{NguyenVeys}. Now, we can use (\ref{gredu}), Proposition \ref{transfer2}, \cite[Subsection 2.2.1 and Theorem 2.10]{VeysZ} and calculate similarly to \cite[Lemma 4.0.1]{Saskia-Kien} to have 
\begin{equation}\label{Jff1}
\int_{\{x\in \cO_L^n, \overline{x}\in Y(k_L), \ord_L(((f)+J_f^2)(x))\geq m-1\}}|dx|\leq C_1m^{n-1}q_L^{-(m-1)\sigma_Y(L, (f)+J_f^2)}
\end{equation}
for all local fields $L\in\cL_{K,M_1}$ and a constant $C_1$ depending only on $n,d$. Here, if $L\in\cL_{K,1}$ then 
$$\sigma_Y(L, (f)+J_f^2)=\min_{Q\in \cO_{L}^n,\overline{Q}\in Y(k_{L})}\lct_Q((f)+J_{f}^2).$$

If $d=2$, we can suppose that $f(0)=0$.  By a suitable change of coordinates, we can suppose that $f=x_1^2+...+x_{n-s}^2$. As in the proof of Proposition \ref{generalZ}, we can suppose that $Y\subset \Sing(f)$ thus we can suppose that $Y=\{0\}\times Y_1$ for some subscheme $Y_1$ of $\AA_{\cO_K}^{s}$. Thus, the assertion follows from (\ref{expdef}).

Let us suppose that $d>2$. We will show that 
\begin{equation}\label{strictbound0}
\sigma_Y(L, (f)+J_f^2)>\frac{n-s}{2(d-1)}
\end{equation} 
for all $L\in\cL_{K,1}$.
By the definition of $\sigma_Y(L, (f)+J_f^2)$, it is sufficient to show that
\begin{equation}\label{strictbound}
\lct_Q((f)+J_{f}^2)>\frac{n-s}{2(d-1)}
\end{equation}
for all $L\in\cL_{K,1}$ and all $Q\in \cO_L^n$. Let $L\in \cL_{K,1}$ and $Q\in \cO_L^n$. Note that if $Q$ is not a critical value of $f$ or $f(Q)\neq 0$ then the left hand side of (\ref{strictbound}) is $+\infty$. Thus, we can suppose that $Q$ is a critical value of $f$ and $f(Q)=0$. 
By adapting the proof of Corollary \ref{uniformlocal2}, we can suppose that $Q=0$ and we have $$\lct_0((f)+J_f^2)\geq \lct_0((f_d)+J_{f_d}^2).$$
Moreover, we also have $\lct_0(J_{f_d}^2)=n/(2(d-1))$ if $\dim(\Sing(f_d))=0$. 
Thus, we only need to prove that
$$\lct_0((f_d)+J_{f_d}^2)>\frac{n}{2(d-1)}=\lct_0(J_{f_d}^2)$$ 
if $\dim(\Sing(f_d))=0$. If this is the case, suppose that $\lct_0((f_d)+J_{f_d}^2)=\lct_0(J_{f_d}^2)=n/(2(d-1))$. Since  $\dim(\Sing(f_d))=0$, if $x\in t\overline{L}[[t]]^n$ then  $\ord_t(J_{f_d}^2(x))\geq m$ iff $\ord_t(x)\geq m/(2(d-1))$. Therefore, if $d>2$ then 
\begin{align*}&\min\bigg\{\ord_t(((f_d)+J_{f_d}^2)(x))\bigg| x\in (t\overline{L}[[t]])^n , \ord_t(J_{f_d}^2(x))\geq m\bigg\}\\=&\min\bigg\{\ord_t(f_d(x))\bigg| x\in (t\overline{L}[[t]])^n,\ord_t(x)\geq \frac{m}{2(d-1)}\bigg\}\\=&d\bigg\lceil\frac{m}{2(d-1)}\bigg\rceil<m
\end{align*}
if $m$ is large enough, where $\lceil a\rceil= \min\{\ell\in \ZZ\mid \ell\geq a\}$ if $a\in \RR$. Therefore,
we can use Proposition \ref{jetlog} to get a contradiction. Thus (\ref{strictbound0}) holds.

By (\ref{strictbound0}) together with the definition of log-canonical thresholds and our assumptions on $h$, there are positive integers $m_0,m_1$ depending only on $n,d$ such that 
\begin{equation}\label{uppm0}m_0\left(\sigma_Y(L, (f)+J_f^2)-\frac{n-s}{2(d-1)}\right)>\sigma_Y(L,(f)+J_f^2)>\frac{1}{m_1}
\end{equation}
for all $L\in \cL_{K,1}$. Thus, there is an integer $M_2$ depending only on $M_1,C_1,m_0,m_1,n,d$ such that 
\begin{equation}\label{upp1}C_1m^{n-1}q_L^{-(m-m_0)\left(\sigma_Y(L, (f)+J_f^2)-\frac{n-s}{2(d-1)}\right)}\leq 1
\end{equation}
for all $L\in\cL_{K,M_2}$ and all $m>m_0$. It follows from (\ref{Jff0}), (\ref{Jff1}), (\ref{uppm0}), (\ref{upp1}) and Proposition \ref{transfer2} that 
\begin{equation}\label{simplify}
|E_{f}(Y,\psi)|\leq q_L^{\frac{-m_\psi(n-s)}{2(d-1)}} 
\end{equation}
for all local fields $L\in\tilde{\cL}_{K,M_2}$ and all additive characters  $\psi$ of $L$ of conductor $m_\psi>m_0$. 

Let $M,C$ be as in Proposition \ref{generalZ} with respect to $Y$. We can finish our proof by using (\ref{simplify}) and putting $M_3=\max\{M,M_2\}$, $C_3=Cm_0^{n-1}$.
\end{proof}

\begin{proof}[Proof of \ref{generalZ2}]
 The assertion follows by combining Proposition \ref{generalZ1} and  \cite[Proof of Proposition 5.4(2)]{NguyenVeys}.
\end{proof}
\begin{cor}\label{globalsums}There exists an integer $M$ depending only on $f$ and a constant  $C$ depending only on $n,d$ such that for all local fields $L\in\tilde{\cL}_{K,M}$  and all additive characters  $\psi$ of $L$ of conductor $m_\psi\geq 1$ we have
$$|E_{f}(\psi)|\leq Cq_L^{\frac{-m_\psi(n-s)}{2(d-1)}}.$$
\end{cor}
\begin{proof}
It follows immediately from \ref{generalZ2} and Proposition \ref{Deligne-Katz}. 
\end{proof}
\begin{rem}
 In Corollary \ref{globalsums}, we cannot take $M$ depending only on $n,d$. Indeed, let us look at the following example.
 
 Let $d>1$.  For each prime number  $p$ and each integer $m\geq 2$, we set $$f_{p,m}(x_1,...,x_n)=x_1^d+p^{m}(x_2^d+...+x_n^d).$$ 
 
Let $2\leq r\leq d$ and $\ell\in\NN$. Then $$E_{f_{p,m}}(\psi)=p^{-\frac{m}{d}}p^{\frac{r-d}{d}}$$ for all additive characters $\QQ_p$ of conductor  $m_\psi=m=d\ell+r$. We remark that $s(f_{p,m})=0$ for all $p,m$. Suppose that $n\geq 2$. It is impossible to have a constant $C$ such that 
$$p^{-\frac{m}{d}}p^{\frac{r-d}{d}}\leq Cp^{\frac{-mn}{2(d-1)}}$$
for all large enough primes $p$ and all positive integers $m$ of the form $d\ell+r$.
\end{rem}
\section{Exponential sums modulo $p^m$ for small $p$ and the strong monodromy conjecture}\label{small}
To begin this section, let us try to show a weaker version of \ref{smallprime2} for $\lct(J_f^2)$ instead of $\lct(f,J_f^2)$.
\begin{prop}\label{smallprime1}Let $L$ be a $p$-adic field. Suppose that $f$ is a non-constant polynomial in $L[x_1,...,x_n]$. For each Schwartz-Bruhat function $\Phi$ on $L^n$ and each $0<\sigma<\frac{\lct_{\Phi}(J_f)}{2}$ there exists a constant $c(f,\sigma,\Phi)>0$ such that for all additive characters $\psi$ of $L$ we have
$$|E_{f}(\Phi,\psi)|\leq c(f,\sigma,\Phi)q_L^{-m_\psi\sigma},$$
where $\lct_{\Phi}(J_f)=\min_{x\in\Supp(\Phi)}\lct_x(J_f)$.
\end{prop}
\begin{proof}Let $\varpi_L$ be a uniformizing
parameter of $\cO_L$. It is clear that
$$|E_{f}(\Phi,\psi)|\leq \int_{L^n}|\Phi(x)||dx|\leq C$$
for a constant $C$ depending only on $\Phi$ and all additive characters $\psi$ of $L$. Therefore it is sufficient to prove the assertion for additive characters $\psi$ of large enough conductor. On the other hand, we can suppose that $\Phi=\textbf{1}_{a+\varpi_L^{e}\cO_L^n}$ for some $a\in L^n$ and $e\in \ZZ$. By a suitable change of coordinates, we can suppose that $a=0$ and $e=0$. Multiplying $f$ by $\varpi_L^r$ if needed, we can suppose that $f\in\cO_L[x_1,...,x_n]$.  

If $f$ has no critical point in $\cO_L^n$, then  $E_{f}(\Phi,\psi)=0$ if $m_\psi$ is large enough. In this case we also have $\lct_{\Phi}(J_f)=+\infty$, so the claim is trivial.

Suppose that $f$ has a critical point in $\cO_L^n$. For each $m\geq 0$, we set $$N_m=|\{x \mod \varpi_L^m\cO_L|x\in \cO_L^n, \ord_L(J_f^2(x))\geq m\}|,$$ where $J_f$ is the Jacobian ideal of $f$ over $\cO_L$ as mentioned in Subsection \ref{setup1} and $\ord_L(J_f^2(x))=\min_{g\in J_f^2}\ord_L(g(x))$.  We also consider the Igusa local zeta function  associated with $J_f^2$ given by
$$\cZ(J_f^2,L,\Phi,s)=\int_{\cO_L^n}q_L^{-\ord_L(J_f^2(x))s}|dx|$$
 It follows  from \cite[Theorem 2.4]{VeysZ} that $\cZ(J_f^2,L,\Phi,s)$ is a rational function in $T=q_L^{-s}$. Moreover,  the real part of any pole of $\cZ(J_f^2,L,\Phi,s)$ is at most $-\lct_\Phi(J_f^2)$ and the multiplicity of any pole of $\cZ(J_f^2,L,\Phi,s)$ is at most $n$. In particular, by \cite[Subsection 2.2.1]{VeysZ}, there is a positive constant $C(f,\Phi)$ depending only on $f,\Phi$ such that 
\begin{equation}\label{uppcong}N_{j}q_L^{-nj}\leq C(f,\Phi)j^{n-1}q_L^{-j\lct_\Phi(J_f^2)}
\end{equation}
for all large enough $j$. By the same argument as in \cite[Proposition 6.4]{CKMu}, we have
\begin{equation}\label{Jacobian}
\int_{\cO_L^n}\psi(f(x))|dx|=\int_{\{x\in \cO_L^n| \ord_L(J_f^2(x))\geq m_\psi-1\}}\psi(f(x))|dx|
\end{equation}
whenever $m_\psi\geq 1$. By combining (\ref{uppcong}) and  (\ref{Jacobian}), we have 
$$|E_{f}(\Phi,\psi)|\leq N_{m_\psi-1}q_L^{-n(m_\psi-1)}\leq C(f,\Phi)m_\psi^{n-1}q_L^{-m_\psi\lct_\Phi(J_f^2)}$$
for all additive characters $\psi$ of $L$ of large enough conductor.

Note that $\lct_\Phi(J_f^2)=\lct_\Phi(J_f)/2$ by \cite[Property 1.13]{Mustata2}. Therefore,  to prove the assertion for all additive characters of $L$ of large enough conductor, it suffices to set $$c(f,\lct_\Phi(J_f)/2-\epsilon,\Phi):= \sup_{m\geq 1}C(f,\Phi)m^{n-1}q_L^{-\epsilon m}$$
for each $0<\epsilon<\lct_\Phi(J_f)/2$.
\end{proof}
\begin{proof}[Proof of \ref{smallprime2}]Let us fix a uniformizing parameter  $\varpi_L$ of $\cO_L$.  We can suppose that $C_f\cap \Supp(\Phi)\subset f^{-1}(0)$ by using the proof of \cite[Proposition 2.7]{DenefVeys}. As in the proof of Proposition \ref{smallprime1}, we can suppose that $\Phi=\textbf{1}_{\cO_L^n}$, $f\in\cO_L[x_1,...,x_n]$ and $f$ has a critical point in $\cO_L^n$. Moreover, we only need to prove the assertion for additive characters of large enough conductor. 

Note that we can adapt the argument in the proof of Proposition \ref{smallprime1} whenever we have a positive integer $r$ independent of additive characters such that if $m_\psi$ is large enough then 
\begin{equation}\label{Jacobian1}
\int_{\cO_L^n}\psi(f(x))|dx|=\int_{\{x\in \cO_L^n| \ord_L(f(x))\geq m_\psi-r\}}\psi(f(x))|dx|.
\end{equation}
In more detail, using (\ref{Jacobian1}), we can use the same discussion as in  \cite[Proposition 5.2]{CKMu} to show that 
\begin{equation}\label{Jacobian2}
\int_{\cO_L^n}\psi(f(x))|dx|=\int_{\{x\in \cO_L^n| \ord_L(J_f^2(x))\geq m_\psi-1, \ord_L(f(x))\geq m_\psi-r\}}\psi(f(x))|dx| 
\end{equation} 
for all additive characters $\psi$ of large enough conductor. Therefore we can bound $|E_{f}(\Phi,\psi)|$ by the volume of the set $$\{x\in\cO_L^n|\ord_L(f(x))\geq m_\psi-r, \ord_L(J_f^2)\geq m_\psi-r\}.$$
Thus we can proceed as in the proof of Proposition \ref{smallprime1} for the ideal $(f)+J_f^2$ instead of $J_f^2$ to obtain our claim.

Let us prove (\ref{Jacobian1}). We fix a log-resolution $h:X'\to\AA_L^n$ of $(f)$ associated with divisors $(E_i)_{i\in \cT}$ and numerical data $(\nu_i,N_i)_{i\in\cT}$. We put $r_0=\max_{i\in\cT}\ord_L(N_i)+1$. Note that $h$ is projective thus $h^{-1}(\cO_L^n)$ is compact. For each point $P\in h^{-1}(\cO_L^n)$, we modify local coordinates at $P$ as in \cite[Section 3, Chapter 3]{Igusa3}. More precisely,  if $P\in h^{-1}(\cO_L^n)$, we have an open compact neighbourhood $U_P$ of $P$ and local coordinates $y_1,...,y_n$ such that the following conditions hold:
\begin{itemize}
\item[i,]$\ord_L(y_i(Q))\geq 0$ for all $Q\in U$ and $1\leq i\leq n$,
\item[ii,]$f\circ h|_U=\eta\prod_{i=1}^ny_i^{N_i}$,
\item[iii,]$h^*(dx_1\wedge...\wedge dx_n)|_U=\vartheta y_1^{\nu_1-1}... y_n^{\nu_n-1}dy_1\wedge...\wedge dy_n$,
\item[iv,]$|\eta|, |\vartheta|$ are non-zero constant functions on $U$,
\item[v,]$\eta(y_1,...,y_n)=\eta(P)(1+ \sum_{I\in \NN^n, |I|\geq 1} \alpha_{I}y^I)$
where $\alpha_{I}\in \cO_L$ for all $I$,
\item[vi,] $\vartheta(y_1,...,y_n)=\vartheta(P)(1+\sum_{I\in\NN^n, |I|\geq 1} \beta_{I}y^I)$
where $\beta_{I}\in \cO_L$ for all $I$.
\end{itemize}
Since $h^{-1}(\cO_L^n)$ is compact, we can select a finite set $\cI$ of points in $h^{-1}(\cO_L^n)$ such that $h^{-1}(\cO_L)\subset \cup_{P\in \cI}U_P$. For distinct points $P, P'$ in $\cI$ then   $U_P\setminus U_{P'}$ is still a compact open set. Therefore, replacing pair $(U_P, U_{P'})$ by pair $(U_P\setminus U_{P'}, U_{P'})$ if needed, we can suppose that $\cup_{P\in\cI}U_P$ is a disjoint union. Replacing $f$ by $\varpi_L^\ell f$ if necessary, we also can suppose that $\eta(P)\in \cO_L$ for all $P\in\cI$. We put $r_1=\max_{P\in\cI}\ord_L(\eta(P))$. 

For each $P\in \cI$, let $y_1,...,y_n$ be the local coordinates on $U_P$ as above. We set $g=(y_1,...,y_n)$.  Since $U_P$ is a compact open set, $g(U_P)$ and $U_P\cap h^{-1}(\cO_L^n)$ are also compact open sets. Thus there exists a positive integer $r_P$ such that if $h(Q)\in \cO_L^n$ and $Q,Q'\in U_P$ then $g(Q)(1+\varpi_L^{r_P}\cO_L)^n\subset g(U_P)$ and $h(Q')\in \cO_L^n$ whenever $ g(Q')\in g(Q)(1+\varpi_L^{r_P}\cO_L)^n$. We put $r_2=\max_{P\in\cI} r_P$ and $r=r_0+r_2$. 

Let $\chi$ be a multiplicative character of $\cO_L^{\times}$ of conductor $c(\chi)>r$. Suppose that $Q\in h^{-1}(\cO_L^n)$ and $f(h(Q))\neq 0$. Let $P\in\cI$ such that $Q\in U_P$. Let $y_1,...,y_n$ be the local coordinates  and $\eta, \vartheta$ be the  functions on $U_P$ as above. We still set $g=(y_1,...,y_n)$. By the definition of $r_2$, we have
\begin{align*}
&\int_{h(g^{-1}(g(Q). (1+\varpi_L^{r_2}\cO_L)^n))}\chi(ac(f(x)))|f(x)|^s|dx_1\wedge...\wedge dx_n|\\
&=\int_{(1+\varpi_L^{r_2}\cO_L)^n}\chi(ac(\eta(y_1(Q)z_1,...,y_n(Q)z_n)\prod_{1\leq i\leq n}(y_i(Q)z_i)^{N_i}))|f(h(Q))|^s\times\\
&|\vartheta(P)|\prod_{1\leq i\leq n}|y_i(Q)|^{\nu_i}|dz_1\wedge...\wedge dz_n|\\
&=\chi(ac(\prod_{1\leq i\leq n}y_i(Q)^{N_i}))|f(h(Q))|^s|\vartheta(P)|\prod_{1\leq i\leq n}|y_i(Q)|^{\nu_i}\times q_L^{-r_2n}\times\\
&\int_{\cO_L^n}\chi(ac(\eta(y_1(Q)(1+\varpi_L^{r_2}t_1),...,y_n(Q)(1+\varpi_L^{r_2}t_n))\prod_{1\leq i\leq n}(1+\varpi_L^{r_2}t_i)^{N_i}))|dt_1\wedge...\wedge dt_n|
\end{align*}
We recall that  $$\eta(y_1,...,y_n)=\eta(P)(1+ \sum_{|I|\geq 1}a_Iy^I)$$ and $a_I\in\cO_L$ for all $I$. 
Suppose that $\ord_L(y_i(Q))\geq r_0$ and $N_i>0$ for some $1\leq i\leq n$. Without loss of generality, we can suppose that $i=1$. Fix  a tuple $(t_1,t_2,...,t_n)\in\cO_L^n$. We have 
\begin{align*}
&\ord_L(\eta(y_1(Q)(1+\varpi_L^{r_2}t_1),...,y_n(Q)(1+\varpi_L^{r_2}t_n))\\&-\eta(y_1(Q)(1+\varpi_L^{r_2}(t_1+\varpi_L^{c(\chi)-r_2-r_0}u)),...,y_n(Q)(1+\varpi_L^{r_2}t_n)))\\&\geq \ord_L(y_1(Q))+c(\chi)-r_0+\ord_L(\eta(P))\\
&\geq c(\chi)+\ord_L(\eta(P))
\end{align*}
for all $u\in\cO_L$. Since $|\eta|$ is constant in $U_P$ we have
\begin{align*}
&\chi(ac(\eta(y_1(P)(1+\varpi_L^{r_2}t_1),...,y_n(P)(1+\varpi_L^{r_2}t_n))))\\&=\chi(ac(\eta(y_1(P)(1+\varpi_L^{r_2}(t_1+\varpi_L^{c(\chi)-r_2-r_0}u)),...,y_n(P)(1+\varpi_L^{r_2}t_n))))
\end{align*}
for all $u\in\cO_L$. Since $\ord_L(N_1)<r_0$, it is easy to verify that 
$$\int_{\cO_L}\chi(ac(1+\varpi_L^{r_2}t_1+\varpi_L^{c(\chi)-r_0}u)^{N_1})|du|=0.$$
 Therefore, the Fubini theorem yields
$$\int_{h(g^{-1}(g(Q). (1+\varpi_L^{r_2}\cO_L)^n))}\chi(ac(f(x)))|f(x)|^s|dx_1\wedge...\wedge dx_n|=0$$
if $c(\chi)>r$, $\ord_L(y_i(P))\geq r_0$ and $N_i>0$ for some $1\leq i\leq n$.

Let $$W=\{b\in \cO_L^n|\ord(f(b))\leq r_1+r_0\sum_{i\in \cT} N_i\}$$
and $V=\cO_L^n\setminus W$. Since $C_f\cap \cO_L^n\subset f^{-1}(0)$, we see that  $W\cap C_f=\emptyset$. Moreover, $W$ is open and compact thus $E_{f}(\textbf{1}_{W},\psi)=0$ if $m_\psi$ is large enough. 

Let $m$ be a large enough integer. We put $V_{m}=\{b\in V| \ord(f(b))<m-r\}$. We will use Proposition \ref{Fourier} to compute $E_{f}(\textbf{1}_{V_{m}},\psi)$ if $m_\psi=m$. Let $b\in V$, $P\in\cI$, $Q\in h^{-1}(b)\cap U_P$ and $y_1,...,y_n$ be the local coordinates on $U_P$ as above.  The definitions of $V$ and $r_1$ imply that there exists $1\leq i\leq n$ such that $N_i>0$ and $\ord_L(y_i(Q))>r_0$. The above discussion together with the compactness of $h^{-1}(V_m)$  and the definition of $r_2$ imply that $\cZ_\chi(f,L,\textbf{1}_{V_m},s)=0$ if $c(\chi)>r$. Therefore, in Proposition \ref{Fourier}, it remains to deal with $\cZ_\chi(f,L,\textbf{1}_{V_{m}},s)$ if $c(\chi)\leq r$.  By the definition of Igusa local zeta function, $\cZ_\chi(f,L,\textbf{1}_{V_{m}},s)$ is a polynomial in $q_L^{-s}$ of degree at most $m-r-1<m-c(\chi)$ if $c(\chi)\leq r$. Therefore Proposition \ref{Fourier} yields $E_{f}(\textbf{1}_{V_{m}},\psi)=0$ if $m_\psi=m$. Thus (\ref{Jacobian1}) is proved.

\end{proof}
\begin{proof}[Proof of \ref{globalfield}]
When $K'=\QQ$, our theorem follows by combining \ref{smallprime2},  Formula (\ref{eq:p-i}), Corollary \ref{globalsums} and the fact that $$\lct_a((f-f(a))+J_f^2)> (n-s)/(2(d-1))$$ for all $a$ as shown in the proof of Proposition \ref{generalZ1}.

When $K'$ is an arbitrary number field, we still have an analogous formula of (\ref{eq:p-i}) as in \cite[Section 1.8]{CluckerNguyen}, thus our claim follows from the same discussion as in the case $K'=\QQ$. 
\end{proof}
\begin{cor}\label{monodromy1}Let $L$ be a $p$-adic field and $f$ be a non-constant polynomial in $L[x_1,...,x_n]$. The strong monodromy conjecture holds in the range $(-\lct((f)+J_f^2),0]$. More precisely, if $\Phi$ is a Schwartz-Bruhat function  on $L^n$, $\chi$ is a multiplicative character of $\cO_L^{\times}$ and $s$ is a pole of $\cZ_{\chi}(f,L,\Phi,s)$  such that $\Re(s)\in(-\lct_\Phi((f)+J_f^2),0]$ then $\Re(s)$ is a root of the Bernstein-Sato polynomial of $f$.
\end{cor}
\begin{proof}The assertion follows by combining \ref{smallprime2} and  the work of Igusa in \cite{Igusa3} (see  also in \cite[Corollary 1.4.5]{DenefBour} or \cite[Proposition 2.7]{DenefVeys}).

\end{proof}
\begin{cor}\label{monodromy2}The assertion of Corollary \ref{monodromy1} holds for the range $[-(n-s)/(2(d-1)),0]$ instead of the range $(-\lct_\Phi((f)+J_f^2),0]$.
\end{cor}
\begin{proof}The claim follows from Corollary \ref{monodromy1} and the proof of Proposition \ref{generalZ1}.
\end{proof}
\section{Estimation for The Major arcs}\label{major}
We recall the formulation of the major arcs in \cite{Birch} and \cite{BrowPren}. Let $f$ be a homogeneous polynomial of degree $d>1$ in $\ZZ[x_1,...,x_n]$ and $s=s(f)$ as in Section \ref{sec:intro}. We suppose that the following condition holds:

\textbf{Local-smooth-soluble (LSS):} The system $$f(x)=0, \grad f(x)=(\frac{\partial f}{\partial x_1}(x),...,\frac{\partial f}{\partial x_n}(x))\neq 0$$ has a solution in $\RR^n$ and a solution in $\FF_p^n$ for each prime number $p$. 

In other words, this means that $f$ has a smooth solution over each completion of $\QQ$.  Let $(C)$ be a condition on $f$. We say that $(C)$ is a sufficient condition for the smooth Hasse principle if for all polynomials $f$ satisfying $(C)$, our condition (\textbf{LSS}) guarantees the existence of a non-trivial solution of $f$ over $\QQ$. We also can ask the smooth Hasse principle in the setting of non-homogeneous polynomials without requiring the non-triviality of solution. 

To study the smooth Hasse principle of $f$, we follow the Hardy-Littlewood  circle method (see \cite{BrowPren}). Let $\omega:\RR^n\to [0,+\infty)$ be a suitable weight function. We will try to give an asymptotic formula for the function
$$N_{\omega}(f,B)=\sum_{x\in\ZZ^n, f(x)=0}\omega(x/B)$$
when $B\rightarrow+\infty$. Let us use the identity
\begin{equation}\label{identity}
N_{\omega}(f,B)=\int_{\TT}S(\alpha,B)d\alpha
\end{equation}
where $\TT=\RR/\ZZ$ and $$S(\alpha,B)=\sum_{x\in\ZZ^n}\omega(x/B)e^{2\pi i\alpha f(x)}$$
if $\omega$ has good enough analytic properties. The Hardy-Littlewood circle method suggests us to divide the torus $\TT$ into a set of major arcs $\mathfrak{M}$ and minor arcs $\mathfrak{m}$. More precisely, for each $\delta>0$ we define the major arcs $\mathfrak{M}=\mathfrak{M}(\delta)$ to be the set
$$\mathfrak{M}(\delta)=\cup_{q\leq B^{\delta}}\cup_{0\leq a\leq q, (a,q)=1}\{\alpha\in \TT| |\alpha-\frac{a}{q}|\leq B^{\delta-d}\}$$
Note that the union is disjoint if  $3\delta<d$ and $B$ is  large enough. We define the minor arcs to be the complement $\mathfrak{m}(\delta)=\TT\setminus\mathfrak{M}(\delta)$.  To deal with the smooth Hasse principle, we want to obtain the following asymptotic formulas
\begin{equation}\label{majorarc}
\int_{\mathfrak{M}}S(\alpha,B)d\alpha\sim c_fB^{n-d}
\end{equation}
and
\begin{equation}\label{minorarc}
\int_{\mathfrak{m}}S(\alpha,B)d\alpha=o(B^{n-d}).
\end{equation}
Here the constant $c_f$ is a product of local densities which will be positive if $f$ satisfies the condition (\textbf{LSS}).   Therefore, to prove (\ref{majorarc}), we need to find conditions on $f$ such that the product of local densities is convergent.  Birch's condition for the existence of (\ref{majorarc}) and (\ref{minorarc}) is $n-s>2^d(d-1)$. This condition is improved by Browning and Prendiville in \cite{BrowPren} to $n-s>\frac{3}{4}2^d(d-1)$ as mentioned in Section \ref{sec:intro}. It seems to be very difficult to improve much on Browning-Prendiville's condition for the validity of (\ref{minorarc}) by the current method related to Weyl bound and Diophantine approximation. However, we can reduce much on the condition for the validity of (\ref{majorarc}) as follows.

\begin{named}{Theorem D}\label{majorarc1}Let $f$ be a non-constant polynomial of degree $d>1$ in $\ZZ[x_1,...,x_n]$. Suppose that $f$ has a smooth solution $x_0\in \RR^n$. We recall $f_d, s$ as in Section \ref{sec:intro}. Suppose that $n-s>4(d-1)$. Let $\omega:\RR^n\to[0,+\infty)$ be a weight function  given by $\omega(x)=w(\rho^{-1}||x-x_0||)$ where $\rho\in (0,1)$ is very closed to $1$, $||.||$  is the Euclidean norm in $\RR^n$
$$||y||^2=\sum_{i=1}^n y_i^2$$
and $w:\RR\to\RR_{\geq 0}$ is given by $x\mapsto 0$ if $|x|\geq 1$ and $x\mapsto e^{\frac{-1}{1-x^2}}$ if $|x|<1$. Let $\delta$ be small enough. Then we have
\begin{equation}\label{majorarc2}
\int_{\mathfrak{M}(\delta)}S(\alpha,B)d\alpha\sim c_{f}B^{n-d}
\end{equation}
for a positive constant $c_{f}$ when $B\to +\infty$  provided that $f$ satisfies the condition \textnormal{(\textbf{LSS})}.
\end{named}
\begin{proof}
Let $\alpha\in\mathfrak{M}(\delta)$ then we have $\alpha=\frac{a}{q}+\theta$ with integers $a,q$ such that $1\leq q\leq B^{\delta}$ and $(a,q)=1$. We use again the argument of \cite[Lemma 5.1]{BrowPren} to obtain:
\begin{equation}\label{eq:major}
S(\alpha,B)=q^{-n}B^nS_q(a)I(\theta B^d)+O(B^{n-1+2\delta})
\end{equation}
where
$$S_q(a)=\sum_{y\in (\ZZ/q\ZZ)^n}e^{\frac{2\pi iaf(y)}{q}}$$
and
$$I(\gamma)=\int_{\RR^n}\omega(x)e^{2\pi i\gamma f(x)}dx$$
if $\gamma\in\RR$. Since the measure of $\mathfrak{M}(\delta)$ is $O(B^{-d+3\delta})$ we have
$$\int_{\mathfrak{M}(\delta)}S(\alpha,B)d\alpha=B^{n-d}\mathfrak{S}(B^{\delta})\mathfrak{I}(B^{\delta})+O(B^{n-d-1+5\delta})$$
where
$$\mathfrak{S}(R)=\sum_{1\leq q\leq R}q^{-n}\sum_{a\in (\ZZ/q\ZZ)^{\times}}S_q(a)$$
and
$$\mathfrak{I}(R)=\int_{-R}^{R}I(\gamma)d\gamma.$$
Let $c_1,...,c_r$ be critical values of $f$ over $\RR$. Let $\varepsilon$ be a positive constant such that $|c_i-c_j|>2\varepsilon$ for all $i\neq j$. For each $1\leq i\leq r$, there exists  a bounded open subset $V_i$ of $\RR^n$ such that $V_i\cap V_j=\emptyset$ for all $i\neq j$ and $U_i=f^{-1}[c_i-\epsilon,c_i+\epsilon]\cap \Supp(\omega)\subset V_i$. For each $1\leq i<r$ there exists a smooth function $u_i$ such that $u_i|_{U_i}=1$, $0\leq u_i \leq 1$ and $u_i|_{\RR^n\setminus V_i}=0$.  We also put $u_{r}=1$. Let $\omega_1:=\omega$ and $\omega_{i+1}=\omega_i-\omega_i u_{i}$ for $1\leq i\leq r-1$. Then $(u_i\omega_i)_{1\leq i\leq r}$ are positive smooth functions with compact support, moreover we have $\omega=\sum_{1\leq i\leq r}\omega_i u_{i}$ and $\Supp(u_i\omega_i)\cap C_f\subset f^{-1}(c_i)$ for all $1\leq i\leq r$. For each $i$, we use \cite[Theorem 1.6]{Igusa3} to have the asymptotic expansion of oscillatory integral associated with the Schwartz function $\omega_iu_i$. Then there exists a constant $\alpha>0$ such that for each small enough $\epsilon>0$, there is a constant $C_\epsilon>0$ such that
$$|I(\gamma)|\leq C_\epsilon |\gamma|^{-\alpha+\epsilon}$$
for large enough $|\gamma|$. So the limit
\begin{equation}\label{Archimedean}
\mathfrak{J}=\lim_{R\rightarrow +\infty}\mathfrak{I}(R)
\end{equation}
exists if one can show that $\alpha>1$. By the work of Igusa in \cite{Igusa3}, $-\alpha$ will be the real part of a certain pole of the function
$$(s+1)\cZ_{1}(\omega_iu_i,f-c_i,s):=\int_{\RR^n}\omega_iu_i|f(x)-c_i|^s|dx|$$
or the function 
$$\cZ_{-1}(\omega_iu_i,f-c_i,s):=\int_{\RR^n}\frac{f(x)-c_i}{|f(x)-c_i|}\omega_iu_i|f(x)-c_i|^s|dx|$$
for some $1\leq i\leq r$. By our assumption, we have $n-s>2(d-1)$.  Thus \ref{generalZ2} and \cite[Proposition 3.10]{CMN} imply that $f-a$ has at worst rational singularities for all critical values $a$ of $f$. Again by the proof of \cite[Proposition 3.10]{CMN}, we can obtain a log-resolution of $(f-a)$ satisfying the condition $(P)$ in \cite[Lemma 6.3, Chapter 4]{Igusa3} for each critical value $a$ of $f$. Therefore $\alpha>1$ by looking at \cite[Theorem 3.1, Chapter 2]{Igusa3}. With this constant $\alpha>1$ and $0<\epsilon<\alpha-1$ we have 
\begin{equation}\label{real}
|\mathfrak{J}-\mathfrak{J}(R)|\leq 2\int_R^{+\infty}C_\epsilon \gamma^{-\alpha+\epsilon}d\gamma= \frac{2C_\epsilon R^{1+\epsilon-\alpha}}{\alpha-1-\epsilon}=O_\epsilon(R^{1+\epsilon-\alpha})
\end{equation}
when $R$ is large enough.

On the other hand, by  \ref{globalfield}, if $n-s>4(d-1)$ then for each $0<\epsilon<\frac{n-s}{2(d-1)}-2$ there exists a constant $C'_\epsilon>0$ such that
$$|S_q(a)|\leq C'_\epsilon q^{n-\frac{n-s}{2(d-1)}+\epsilon}$$
for all $q\geq 1$. Therefore
\begin{align*}
|\sum_{q>R}q^{-n}\sum_{a\in(\ZZ/q\ZZ)^{\times}}S_{q}(a)|&\leq C'_\epsilon\sum_{q>R}q^{1-\frac{n-s}{2(d-1)}+\epsilon}\\
&=O_\epsilon(R^{2-\frac{n-s}{2(d-1)}+\epsilon})
\end{align*}
Thus $\mathfrak{S}(R)$ converges to the singular series
$$\mathfrak{S}=\sum_{1\leq q}q^{-n}\sum_{a\in (\ZZ/q\ZZ)^{\times}}S_q(a)$$
when $R\rightarrow +\infty$. Moreover we have
\begin{equation}\label{padicf}
|\mathfrak{S}-\mathfrak{S}(B^{\delta})|=O_\epsilon(B^{\delta(2-\frac{n-s}{2(d-1)}+\epsilon)})
\end{equation}
Let $\epsilon$ and $\delta$ be small enough. By (\ref{real}) and (\ref{padicf}), we obtain
$$\int_{\mathfrak{M}(\delta)}S(\alpha,B)d\alpha=\mathfrak{S}\mathfrak{I}
B^{n-d}+o(B^{n-d-\vartheta(\epsilon,\delta)})$$
for large enough $B$, where $\vartheta(\epsilon,\delta)$ is a positive constant depending only on $\epsilon,\delta$. Now we can look at the argument in \cite[Section 6, Section 7]{Birch} to see that $\mathfrak{I}>0$ whenever $f$ has a smooth solution in the support of $\omega$ and
$$\mathfrak{S}=\prod_{p \textnormal{ primes}}\mathfrak{S}_p$$
is the convergent product of the local densities $\mathfrak{S}_p$, where
$$\mathfrak{S}_p=\sum_{r\geq 0}p^{-rn}\sum_{a\in(\ZZ/p^r\ZZ)^{\times}}S_{p^r}(a)$$
is positive whenever $f\mod p$ has a smooth solution in $\FF_p^n$.
\end{proof}
\section*{Acknowledgement} 
I would like to thank Jan Denef, François Loeser, Mircea Musta\c{t}\u{a},  Mathias Stout, Floris Vermeulen and Willem Veys for inspiring discussions on the topics of this paper. I am grateful to Raf Cluckers for his guidance together with many useful comments and suggestions during the preparation of this paper. I am partially supported by Fund for Scientific Research-Flanders (Belgium) (F.W.O.) 12X3519N and 1270923N, KU Leuven grant C14/17/083, and the Excellence Research Chair “FLCarPA: L-functions in positive characteristic and applications” financed by the Normandy Region. I am also funded by Vingroup Joint Stock Company and supported by Vingroup Innovation Foundation (VinIF) under the project code VINIF.2021.DA00030.
 \bibliographystyle{amsplain}
\bibliography{anbib}

\end{document}